\documentclass[microtype]{gtpart}
\usepackage[utf8]{inputenc}
\gtart

\usepackage[all]{xy}
\usepackage{mathrsfs, enumitem}

\title[Affine representability results in $\aone$--homotopy theory II]{Affine representability results in $\aone$--homotopy theory II: principal bundles and homogeneous spaces}

\author{Aravind Asok}
\givenname{Aravind}
\surname{Asok}
\address{Department of Mathematics, University of Southern California, Los Angeles, CA 90089-2532, United States}
\email{asok@usc.edu}

\author{Marc Hoyois}
\givenname{Marc}
\surname{Hoyois}
\address{Department of Mathematics, Massachusetts Institute of Technology, Cambridge, MA 02139-4307, United States}
\email{hoyois@mit.edu}

\author{Matthias Wendt}
\givenname{Matthias}
\surname{Wendt}
\address{Fakult\"at f\"ur Mathematik, Universit\"at Duisburg-Essen, Thea-Leymann-Strasse 9, 45127 Essen, Germany}
\email{matthias.wendt@uni-due.de}


\newcommand{\tensor}{\otimes}
\newcommand{\colim}{\operatorname{colim}}

\newcommand{\holim}{\operatorname{holim}}
\newcommand{\Spec}{\operatorname{Spec}}

\newcommand{\isomto}{{\stackrel{\sim}{\;\longrightarrow\;}}}
\newcommand{\isomt}{{\stackrel{{\scriptscriptstyle{\sim}}}{\;\rightarrow\;}}}

\newcommand{\sma}{{\scriptstyle{\wedge}}}

\renewcommand{\O}{{\mathcal O}}
\renewcommand{\hom}{\operatorname{Hom}}

\renewcommand{\Z}{{\mathbb Z}}

\newcommand{\A}{{\mathbb A}}
\newcommand{\I}{{\mathrm I}}

\newcommand{\aone}{{\mathbb A}^1}
\newcommand{\pone}{{\mathbb P}^1}

\newcommand{\gm}[1]{{{\mathbf G}_{m}^{#1}}}

\newcommand{\ho}[1]{\mathscr{H}({#1})}

\newcommand{\fppf}{\operatorname{fppf}}
\newcommand{\et}{\mathrm{\acute et}}
\newcommand{\Nis}{\operatorname{Nis}}
\newcommand{\Zar}{\operatorname{Zar}} % Zariski

\newcommand{\Sm}{\mathrm{Sm}}

\newcommand{\F}{{\mathscr F}}

\newcommand{\Sing}{\operatorname{Sing}}
\newcommand{\Singaone}{\operatorname{Sing}^{\aone}\!\!}

\newcommand{\aff}{\mathit{aff}}

\theoremstyle{plain}
\newtheorem{thm}{Theorem}[subsection]

\newtheorem{lem}[thm]{Lemma}
\newtheorem{cor}[thm]{Corollary}
\newtheorem{prop}[thm]{Proposition}
\newtheorem*{claim*}{Claim}  %Claim

\newtheorem*{thm*}{Theorem}
\newtheorem*{problem*}{Problem}

\newtheorem{thmintro}{Theorem}

\theoremstyle{definition}
\newtheorem{defn}[thm]{Definition}

\newtheorem{rem}[thm]{Remark}
\newtheorem{remintro}[thmintro]{Remark}

\newtheorem{ex}[thm]{Example}

\numberwithin{equation}{section}

\begin{document}

\begin{abstract}
We establish a relative version of the abstract ``affine representability'' theorem in ${\mathbb A}^1$--homotopy theory from Part I of this paper.  We then prove some ${\mathbb A}^1$--invariance statements for generically trivial torsors under isotropic reductive groups over infinite fields analogous to the Bass--Quillen conjecture for vector bundles.  Putting these ingredients together, we deduce representability theorems for generically trivial torsors under isotropic reductive groups and for associated homogeneous spaces in ${\mathbb A}^1$--homotopy theory.
\end{abstract}

\maketitle

\section{Introduction}
Suppose $k$ is a fixed commutative unital base ring, and write $\ho{k}$ for the Morel--Voevodsky $\aone$--homotopy category over $k$ \cite{MV}.  The category $\ho{k}$ is constructed as a certain localization of the category of simplicial presheaves on $\Sm_k$, the category of smooth $k$--schemes.  Write $\Sm_k^{\aff}$ for the subcategory of $\Sm_k$ consisting of affine schemes.  If $\mathscr{X}$ is a simplicial presheaf on $\Sm_k$, by an ``affine representability'' result for $\mathscr{X}$, we will mean, roughly, a description of the presheaf on $\Sm_k^{\aff}$ defined by $U \mapsto [U,\mathscr{X}]_{\aone} := \hom_{\ho{k}}(U,X)$.

Here is a flavor of the description we provide: if $\mathscr{X}$ is a simplicial presheaf on $\Sm_k$, then for any $U \in \Sm_k^{\aff}$ one can consider the simplicial set $\Singaone \mathscr{X}(U)$ \cite[p. 87]{MV}.  The $0$--simplices of this simplicial set are morphisms $U \to \mathscr{X}$ and the $1$--simplices are ``naive'' or ``elementary'' $\aone$--homotopies $U \times \aone \to \mathscr{X}$.  The assignment $U \mapsto \pi_0(\Singaone \mathscr{X}(U))$ defines a presheaf $\pi_0(\Singaone \mathscr{X})$ of ``naive'' $\aone$--homotopy classes of maps $U \to \mathscr{X}$.  In \cite{PartI}, we gave conditions that allowed us to identify $\pi_0(\Singaone \mathscr{X})(U) \cong [U,\mathscr{X}]_{\aone}$, i.e., under which ``naive'' $\aone$--homotopy classes coincide with ``true'' $\aone$--homotopy classes.

In \cite[Theorem 1]{PartI}, building on results of M Schlichting \cite[Theorems 6.15 and 6.22]{Schlichting}, we simplified and generalized F Morel's affine representability result for vector bundles; we encourage the reader to consult the introduction of \cite{PartI} for a more detailed discussion of these points.  Our goal in this paper is to further extend the scope of these affine representability results in $\aone$--homotopy theory.  For example, the following result provides a generalization of the representability result from vector bundles to torsors under suitable reductive group schemes (the description in terms of naive homotopy classes is hidden here).

\begin{thmintro}[See Theorem~\ref{thm:isotropictorsors}]
\label{thmintro:isotropictorsors}
Suppose $k$ is an infinite field, and $G$ is an isotropic reductive $k$--group (see \textup{Definition~\ref{defn:isotropic}}).  For every smooth affine $k$--scheme $X$, there is a bijection
\[
H^1_{\Nis}(X,G) \cong [X,BG]_{\aone}
\]
that is functorial in $X$.
\end{thmintro}

\begin{remintro}
Theorem~\ref{thmintro:isotropictorsors} is essentially the strongest possible representability statement for which one could hope.  First, one cannot expect the functor ``isomorphism classes of Nisnevich locally trivial $G$--torsors'' to be representable on $\ho{k}$ in general.  Indeed, if we do not restrict attention to the category $\Sm_k^{\aff}$, then this functor need not even be $\aone$--invariant (see, e.g., Ramanathan \cite{Ramanathan} for a study of failure of homotopy invariance in case $X = \pone$ or the introduction to \cite{PartI} for other ways in which $\aone$--invariance can fail).  Second, at least if $k$ is infinite and perfect, then the hypothesis that $G$ is isotropic cannot be weakened.  Indeed, if $G$ is not an isotropic reductive $k$--group in the sense mentioned above, then even affine representability for $G$--torsors fails in general; see Remark~\ref{rem:hoinvfailure} and Balwe--Sawant \cite[Theorem 1]{BalweSawantReductive} for more details.  We do not know if Theorem~\ref{thmintro:isotropictorsors} holds if $k$ is finite.
\end{remintro}

\begin{remintro}
\label{remintro:specialgroups}
It has been known for some time that an analogue of Morel's theorem should hold for torsors under groups like $SL_n$ and $Sp_{2n}$ (for $SL_n$ this is mentioned, e.g., in Asok--Fasel \cite[Theorem 4.2]{AsokFaselSpheres}).  Schlichting observed \cite[Remark 6.23]{Schlichting} that his techniques also apply to torsors under groups like $SL_n$ or $Sp_{2n}$.  Combined with the results of \cite{PartI}, one therefore expects affine representability results to hold for torsors under such groups in the same generality as for vector bundles.  For completeness, we include such results here as Theorems~\ref{thm:orientedrepresentability} and \ref{thm:symplecticrepresentability}.
\end{remintro}

We also establish affine representability results for various homogeneous spaces under reductive groups.

\begin{thmintro}[See Theorem~\ref{thm:parabolics}]
\label{thmintro:parabolics}
Suppose $k$ is an infinite field, and $G$ is an isotropic reductive $k$--group.  If $P \subset G$ is a parabolic $k$--subgroup possessing an isotropic Levi $k$--subgroup, then for any smooth affine $k$--scheme $X$, there is a bijection
\[
\pi_0(\Singaone G/P)(X) \isomto [X,G/P]_{\aone}
\]
that is functorial in $X$.
\end{thmintro}

\begin{remintro}
As suggested prior to the statement, we actually establish representability results with targets that are more general homogeneous spaces.  In this direction, observe that it is often possible to ``explicitly'' identify sets of naive homotopy classes and thus, via Theorem \ref{thmintro:parabolics} true $\aone$--homotopy classes. Barge and Lannes \cite[Chapter 4]{BargeLannes} provide such identifications in the case where the target is related to symmetric bilinear forms.  Cazanave \cite{Cazanave} provides such identifications in the case where the target is ${\mathbb P}^n$.  In addition, Fasel \cite[Theorem 2.1]{FaselUnimodular} gives such an identification in the case where the target is a Stiefel variety (various homogeneous spaces of $GL_n$).
\end{remintro}

Building on the ideas of Schlichting and Morel, the proofs of the results above are established using the framework developed in \cite{PartI}: affine representability follows from affine Nisnevich excision and affine homotopy invariance.  The restrictions on $k$ that appear in our results are imposed to guarantee that affine homotopy invariance holds for Nisnevich locally trivial torsors under $G$.

While affine homotopy invariance for vector bundles is precisely the Bass--Quillen conjecture (about which much is known), precise statements regarding affine homotopy invariance for torsors under other groups are harder to find in the literature (in part because such results are typically false for \'etale locally trivial torsors), but see Wendt \cite[Section 3]{WendtRationallyTrivial}.  The entirety of Section~\ref{sec:groupstorsors} is devoted to studying affine homotopy invariance for torsors under reductive group schemes over a rather general base.

Theorem~\ref{thmintro:isotropictorsors} is a straightforward consequence of our general representability result (see Theorem~\ref{thm:Gtorsors}) combined with affine homotopy invariance (see Theorem~\ref{thm:hoinvisotropic} for a precise statement of what we mean by this term).    Theorem~\ref{thmintro:parabolics} follows from Theorem~\ref{thm:homogeneousrep} and affine homotopy invariance for isotropic reductive $k$--groups by a reduction from $P$ to a Levi factor of $P$ (which by assumption is also an isotropic reductive $k$--group).  Again, for certain groups, significantly more general statements can be made; see Theorem \ref{thm:homogeneousspacesunderspecialgroups}.

Our techniques also allow us to establish significant generalizations (with simpler proofs) of some results of F Morel regarding when classifying spaces for split groups are $\aone$--local \cite[Theorems 1.3, 1.5 and A.2]{MFriedMil}.  While Morel deduces these results from strong $\aone$--invariance of non-stable $K_1$--functors, which he establishes by appeal to classical results regarding elementary matrices, we are, in sharp contrast, able to deduce such strong $\aone$--invariance statements as a direct consequence of our general representability result (see Theorem~\ref{thm:wha1invariance} for more details).  As another sample application of these results, we adapt some classical ideas of G\,W Whitehead \cite{Whitehead} to establish nilpotence results for non-stable $K_1$--functors (see Theorem~\ref{thm:nilpotentbyabelian}), along the lines of the results of Bak \cite{BakNilpotent} and Bak--Hazrat--Vavilov \cite{BakHazratVavilov}.  In particular, we are able to resolve \cite[Problem 6]{BakHazratVavilov} in a number of new situations (see Remark~\ref{rem:bhvproblem6} for more details).

The representability results for homogeneous spaces are relevant when applying the methods of obstruction theory to analyze algebraic classification problems. For example, if the base $k$ is a perfect field, the $\aone$--fibration sequence
\[
\mathbb{A}^n\setminus\{0\} \longrightarrow BGL_{n-1} \longrightarrow BGL_n
\]
was used by F Morel \cite[Chapter 8]{MField} to develop an obstruction-theoretic approach to answering the question of when a vector bundle over a smooth affine variety splits off a trivial rank $1$ summand; this approach was further developed by the first author and Fasel in \cite{AsokFaselThreefolds, AsokFaselA3minus0} to which we refer the interested reader for a more detailed discussion.  The results of this paper (specifically Theorem~\ref{thm:A1fibration}) open the possibility of studying such questions over more general base rings, e.g., $\Z$.

Our representability results also broaden the scope of geometric and algebraic applications of $\aone$--homotopy theory.  We mention a few such directions here (though we do not develop the applications).  First, Theorem \ref{thmintro:isotropictorsors} allows one to give explicit classifications of principal $G$--bundles on certain quadric hypersurfaces, see Asok--Fasel \cite{AsokFaselSpheres} and Asok--Doran--Fasel \cite{AsokDoranFasel}.  Theorems~\ref{thm:oddquadric} and \ref{thm:evenquadric} establish affine representability results for ``split'' quadric hypersurfaces.  The former result has relevance to questions regarding unimodular rows (see \cite{AsokFaselSpheres}).  Building on the ideas of Fasel \cite{FaselComplete}, affine representability results for even dimensional quadrics are a key tool in Asok--Fasel \cite{AsokFaselcohomotopy} to interpret Euler class groups \`a la Bhatwadekar--Sridharan in terms of $\aone$--homotopy theory.  In another direction, since the homogeneous space $G_2/SL_3$ is a $6$--dimensional ``split'' smooth affine quadric, we use our results in \cite{AHWOctonion} to study questions regarding reductions of structure group for ``generically trivial'' octonion algebras.  In algebraic terms this can be rephrased as follows: when is an octonion algebra a Zorn (``vector-matrix'') algebra (see, e.g., Springer and Veldkamp \cite[p. 19]{SpringerVeldkamp})?

\subsubsection*{Dependency of sections/prerequisites}
Section~\ref{sec:representability} is devoted to extending our results from \cite{PartI}; the proofs rely on ideas from {\em loc. cit}, which we will use rather freely together with some basic properties of torsors and homogeneous spaces collected in Sections~\ref{ss:torsorsrepresentable} and \ref{ss:applicationhomogeneousspaces}.  Section~\ref{sec:groupstorsors} is devoted to establishing affine homotopy invariance results for torsors under reductive groups.  The results of this section rely on the basic properties of torsors and homogeneous spaces recalled in Section~\ref{sec:representability} as well as the theory of (reductive) group schemes over a base; regarding the latter: we review some of the main definitions and basic properties, but we mainly provide pointers to the literature.  At the very end of Section~\ref{ss:homotopyinvariance} we also rely on the representability results from Section~\ref{sec:representability}.  Section~\ref{sec:applications} contains applications of our main results and thus relies on all of the preceding sections.  We refer the reader to the beginning of each section for a more detailed description of its contents.

\subsubsection*{Acknowledgements}
The authors would like to thank Brian Conrad for extremely helpful correspondence regarding \cite{Conrad} and \cite{ConradGabberPrasad}.  In particular, the proof of Lemma~\ref{lem:parabolics} in its current form was a product of these discussions.  The authors would also like to thank Chetan Balwe and Anand Sawant for helpful discussions of \cite{BalweSawant}, Philippe Gille and Anastasia Stavrova for helpful comments and corrections on a previous version of this paper, and Marco Schlichting and Jean Fasel for mentioning at some point the Zariski local triviality of the torsor appearing in the proof of Lemma \ref{lem:orthogonalquotient}.  Finally, this paper owes an intellectual debt to Marco Schlichting: even if it is obscured in references to \cite{PartI}, the ideas of \cite[\S 6]{Schlichting} served to focus our attention (for example, he established a representability result for special linear groups or symplectic groups \cite[Remark 6.23]{Schlichting}).

Aravind Asok was partially supported by National Science Foundation Award DMS-1254892. Marc Hoyois was partially supported by National Science Foundation Award DMS-1508096. Matthias Wendt was partially supported by EPSRC grant EP/M001113/1.

\subsubsection*{Preliminaries/Notation}
All rings considered in this paper will be assumed unital.  We use the symbol $S$ for a quasi-compact, quasi-separated base scheme, $\Sm_S$ for the category of finitely presented smooth $S$--schemes, and $\Sm_S^\aff\subset\Sm_S$ for the full subcategory of affine schemes (in the absolute sense). We also reuse some terminology and notation introduced in \cite{PartI}, e.g., the notion of affine Nisnevich excision \cite[Example 2.1.2 and Definition 3.2.1]{PartI}, the $t$--localization functor $R_t$ \cite[\S 3.1]{PartI}, the singular construction $\Sing^\I$ \cite[\S 4.1]{PartI}, etc.

\section{Some general representability results}
\label{sec:representability}
The goal of this section is to extend the affine representability results of \cite{PartI}.  In particular, Theorem~\ref{thm:fiberaffinerepresentability} provides a relative version of \cite[Theorem 5.1.3]{PartI}.  We then specialize this result to two cases of particular interest in Theorems~\ref{thm:Gtorsors} and \ref{thm:homogeneousrep}.

\subsection{Naive $\A^1$--homotopy classes}

Let $\mathscr F$ be a simplicial presheaf on $\Sm_S$. Given $X\in\Sm_S$, there is a canonical map
\begin{equation}\label{eqn:naiveVSgenuine}
\pi_0(\Singaone\mathscr F)(X) \to [X,\mathscr F]_{\A^1},
\end{equation}
where the right-hand side is the set of maps in the $\A^1$--homotopy category $\ho{S}$. The left-hand side is the set of \emph{naive} $\A^1$--homotopy classes of maps from $X$ to $\mathscr F$: it is the quotient of the set of maps $X\to\mathscr F$ by the equivalence relation generated by $\A^1$--homotopies.
For presheaves $\mathscr F$ of ``geometric origin'', such as representable presheaves, it is rare that~\eqref{eqn:naiveVSgenuine} is a bijection for all $X\in\Sm_S$ (this happens for example when $\mathscr F$ is represented by an $\aone$--rigid smooth scheme in the sense of Morel--Voevodsky \cite[\S 3 Example 2.4]{MV}, e.g., a smooth curve of genus $g > 0$ or an abelian variety). However, one of the main themes of this paper is that there are many examples of presheaves $\mathscr F$ such that~\eqref{eqn:naiveVSgenuine} is a bijection for every \emph{affine} $X$. We formalize this idea in the following definition.

\begin{defn}
\label{defn:aonenaive}
Let $\mathscr F$ be a simplicial presheaf on $\Sm_S$ and let $\tilde{\mathscr F}$ be a Nisnevich-local $\A^1$--invariant fibrant replacement of $\mathscr F$. Then there is a canonical map $\Singaone\mathscr F\to\tilde{\mathscr F}$, well-defined up to simplicial homotopy. We will say that $\mathscr F$ is \emph{$\A^1$--naive} if the map $\Singaone\mathscr F(X) \to \tilde{\mathscr F}(X)$ is a weak equivalence for every $X\in\Sm_S^\aff$.
\end{defn}

\begin{rem}
\label{rem:homotopygroups}
If $\mathscr F$ is $\A^1$--naive, then in particular~\eqref{eqn:naiveVSgenuine} is a bijection for every $X\in\Sm_S^\aff$. More generally, if $\mathscr F$ is $\A^1$--naive and pointed, then
\[
\pi_n(\Singaone\mathscr F)(X) \cong [S^n\wedge X_+,\mathscr F]_{\A^1,*}
\]
for every $X\in\Sm_S^\aff$ and $n\geq 0$.
\end{rem}

\begin{prop}\label{prop:A1naive}
If $\mathscr F$ is a simplicial presheaf on $\Sm_S$, then $\mathscr F$ is $\A^1$--naive if and only if $\Singaone\mathscr F$ satisfies affine Nisnevich excision (see \textup{\cite[\S2.1]{PartI}}). In that case, $R_{\Zar}\Singaone\mathscr F$ is Nisnevich-local and $\A^1$--invariant.
\end{prop}

\begin{proof}
	Let $\tilde{\mathscr F}$ be a Nisnevich-local $\A^1$--invariant replacement of $\mathscr F$.
	Suppose that $\mathscr F$ is $\A^1$--naive. Then the restriction of $\Singaone\mathscr F$ to $\Sm_S^\aff$ is (objectwise) weakly equivalent to $\tilde{\mathscr F}$, and hence it is Nisnevich-local. But this implies that $\Singaone\mathscr F$ satisfies affine Nisnevich excision, by \cite[Theorem 3.2.5]{PartI}.
	
	Conversely, suppose that $\Singaone\mathscr F$ satisfies affine Nisnevich excision. By \cite[Theorem 3.3.4]{PartI}, the canonical map
	\[
	\Singaone\mathscr F(X) \to R_{\Zar}\Singaone\mathscr F(X)
	\]
	is a weak equivalence for every $X\in\Sm_S^\aff$, and $R_{\Zar}\Singaone\mathscr F$ is Nisnevich-local. By \cite[Lemma 5.1.2]{PartI}, $R_{\Zar}\Singaone\mathscr F$ is also $\A^1$--invariant. Hence, $R_{\Zar}\Singaone\mathscr F\simeq\tilde{\mathscr F}$ and $\mathscr F$ is $\A^1$--naive.
\end{proof}

\subsection{The singular construction and homotopy fiber sequences}
The notion of representable interval object was formulated in \cite[Definition 4.1.1]{PartI}. By a homotopy fiber sequence of pointed simplicial presheaves, we mean a homotopy Cartesian square in which either the top-right or bottom-left corner is a point.

\begin{prop}
\label{prop:fiberSequence}
Let $\mathbf{C}$ be a small category and $\I$ a representable interval object in $\mathbf C$. Let
\[
\mathscr F \longrightarrow \mathscr G \longrightarrow \mathscr H
\]
be a homotopy fiber sequence of pointed simplicial presheaves on $\mathbf C$. If $\pi_0(\mathscr H)$ is $\I$--invariant, then
\[
\Sing^\I\mathscr F \longrightarrow \Sing^\I\mathscr G \longrightarrow \Sing^\I\mathscr H
\]
is a homotopy fiber sequence.
\end{prop}

\begin{proof}
	For $X\in \mathbf{C}$, consider the square of bisimplicial sets
	\[
	\xymatrix{
	\mathscr F(X\times \I^\bullet) \ar[r] \ar[d] & \mathscr G(X \times \I^{\bullet}) \ar[d] \\
	\ast \ar[r] & \mathscr H(X\times \I^\bullet)
	}
	\]
	which is degreewise homotopy Cartesian. Since $\pi_0(\mathscr H)$ is $\I$--invariant, the simplicial set $\pi_0\mathscr H(X\times \I^\bullet)$ is constant. By \cite[Lemma 4.2.1]{PartI}, the diagonal of this square is homotopy Cartesian, i.e.,
\[
\Sing^\I\mathscr F(X) \longrightarrow \Sing^\I\mathscr G(X) \longrightarrow \Sing^\I\mathscr H(X)
\]
is a homotopy fiber sequence.
\end{proof}

\begin{cor}
\label{cor:SingOmega}
Let $\mathbf{C}$ be a small category and $\I$ a representable interval object in $\mathbf C$. If $\F$ is a pointed simplicial presheaf on $\mathbf{C}$ such that $\pi_0(\F)$ is $\I$--invariant, then the canonical map
	\[
	\Sing^\I\mathbf{R}\Omega \F \longrightarrow \mathbf{R}\Omega \Sing^\I \F
	\]
	is a weak equivalence.
\end{cor}

\begin{proof}
This follows from Proposition~\ref{prop:fiberSequence} applied to the homotopy fiber sequence $\mathbf R\Omega(\F) \to * \to \F$.
\end{proof}

\begin{lem}
\label{lem:excisionandfibersequences}
Suppose $\mathbf{C}$ is a small category with an initial object and let $P$ be a cd-structure on $\mathbf{C}$.  If $\mathbf{J}$ is a small diagram and $F: \mathbf{J} \to \mathrm{sPre}(\mathbf{C})$ is a functor such that $F(j)$ satisfies $P$--excision for every $j \in J$, then $\holim_{\mathbf{J}} F$ satisfies $P$--excision as well.
\end{lem}

\begin{proof}
This is a straightforward consequence of commutation of homotopy limits.
\end{proof}

\begin{thm}
\label{thm:fiberaffinerepresentability}
Suppose
\[
\mathscr F \longrightarrow \mathscr G \longrightarrow \mathscr H
\]
is a homotopy fiber sequence of pointed simplicial presheaves on $\Sm_S$.  If the following conditions hold:
\begin{enumerate}
	\item[(i)] $\mathscr G$ and $\mathscr H$ satisfy affine Nisnevich excision, and 
	\item[(ii)] $\pi_0(\mathscr G)$ and $\pi_0(\mathscr H)$ are $\aone$--invariant on affine schemes, 
\end{enumerate}
then $\mathscr F$ is $\A^1$--naive.
\end{thm}

\begin{proof}
By Proposition~\ref{prop:fiberSequence}, for every $U\in\Sm_S^\aff$, the sequence
\begin{equation}
\label{eqn:fibersequence}
\Singaone\mathscr F(U) \longrightarrow \Singaone\mathscr G(U) \longrightarrow \Singaone\mathscr H(U)
\end{equation}
is a homotopy fiber sequence. By \cite[Corollary 4.2.4]{PartI}, both $\Singaone\mathscr G$ and $\Singaone\mathscr H$ satisfy affine Nisnevich excision. Hence by Lemma~\ref{lem:excisionandfibersequences}, $\Singaone\mathscr F$ also satisfies affine Nisnevich excision. In other words, by Proposition~\ref{prop:A1naive}, $\mathscr F$ is $\A^1$--naive.
\end{proof}

The following result is not used in the sequel, but it fits the theme of this section. It is a variant of a result of Morel \cite[Theorem 6.53]{MField} that holds over arbitrary base schemes.

\begin{thm}\label{thm:A1fibration}
Let $\mathscr F\to\mathscr G\to \mathscr H$ be a homotopy fiber sequence of pointed simplicial presheaves on $\Sm_S$. Assume that:
\begin{enumerate}
	\item[(i)] $\mathscr H$ satisfies affine Nisnevich excision;
	\item[(ii)] $\pi_0(\mathscr H)$ is $\aone$--invariant on affine schemes.
\end{enumerate}
Then $\mathscr F\to\mathscr G\to \mathscr H$ is an $\A^1$--fiber sequence, i.e., it remains a homotopy fiber sequence after taking Nisnevich-local $\A^1$--invariant replacements.
\end{thm}

\begin{proof}
	As in Theorem~\ref{thm:fiberaffinerepresentability}, the sequence~\eqref{eqn:fibersequence} is a homotopy fiber sequence for every $U\in\Sm_S^\aff$.
	Let $i^*$ be the restriction functor from $\Sm_S$ to $\Sm_S^\aff$ and $\mathbf Ri_*$ its derived right adjoint. By \cite[Lemma 3.3.2]{PartI}, there is a natural equivalence of functors $R_{\Zar}\simeq \mathbf Ri_*R_{\Zar}i^*$.	Since $\mathbf Ri_*$ and $R_{\Zar}$ preserve homotopy fiber sequences, we deduce that
	\[
	R_{\Zar}\Singaone\mathscr F \longrightarrow R_{\Zar}\Singaone\mathscr G \longrightarrow R_{\Zar}\Singaone\mathscr H
	\]
	is a homotopy fiber sequence. By \cite[Theorem 5.1.3]{PartI}, $R_{\Zar}\Singaone\mathscr H$ is Nisnevich-local and $\A^1$--invariant.
	But it follows from the right properness of the Morel--Voevodsky model structure \cite[\S2 Theorem 2.7]{MV} that every homotopy fiber sequence whose base is Nisnevich-local and $\A^1$--invariant is an $\A^1$--fiber sequence.
\end{proof}

\subsection{Application to torsors}
\label{ss:torsorsrepresentable}
In this subsection we specialize the general representability result of \cite[\S5.1]{PartI} to simplicial presheaves classifying $G$--torsors for some group $G$. We start by recalling some general facts about torsors.

\begin{defn}\label{def:torsor}
	Let $\mathbf C$ be a small category equipped with a Grothendieck topology $t$, let $G$ be a $t$--sheaf of groups on $\mathbf C$, and let $X\in\mathbf C$. A \emph{$G$--torsor over $X$} is a triple $(\mathscr P,\pi,a)$ where $\mathscr P$ is a $t$--sheaf on $\mathbf C$, $a\colon\mathscr P\times G\to\mathscr P$ is a right action of $G$ on $\mathscr P$, and $\pi\colon \mathscr P\to X$ is a morphism that is $G$--equivariant for the trivial $G$--action on $X$, such that:
	\begin{enumerate}
		\item[(i)] the morphism $\mathscr P\times G \to \mathscr P\times_X\mathscr P$ of components $\pi_1$ and $a$ is an isomorphism;
		\item[(ii)] $\pi$ is $t$--locally split, i.e., the collection of morphisms $U\to X$ in $\mathbf C$ such that $\mathscr P\times_XU\to U$ has a section is a $t$--covering sieve of $X$.
	\end{enumerate}
\end{defn}

The collection of $G$--torsors over various $X\in\mathbf C$ can be assembled into a category $\mathbf{Tors}_t(G)$ fibered in groupoids over $\mathbf C$. We write $B\mathbf{Tors}_t(G)$ for the simplicial presheaf whose value on $U\in\mathbf C$ is the nerve of the groupoid of sections of $\mathbf{Tors}_t(G)$ over $\mathbf{C}/U$ (this groupoid is canonically equivalent to the groupoid of $G$--torsors over $U$, but is strictly functorial in $U$, cf. Hollander \cite[\S3.3]{Hollander}).
It is well-known that $\mathbf{Tors}_t(G)$ is a stack for the topology $t$. As shown in \cite[Theorem 3.9]{Hollander}, this is equivalent to the statement that $B\mathbf{Tors}_t(G)$ satisfies $t$--descent.

We denote by $BG$ the pointed simplicial presheaf with $n$--simplices $G^n$ and with the usual face and degeneracy maps, and we let
\[
B_tG:=R_t BG
\]
be its $t$--local replacement (see \cite[\S 3]{PartI}). There is a morphism $BG\to B\mathbf{Tors}_t(G)$ sending the unique vertex of $BG(U)$ to the trivial $G$--torsor over $U$. Since $B\mathbf{Tors}_t(G)$ is $t$--local, we obtain a morphism of simplicial presheaves
\begin{equation}
\label{eqn:BG}
B_tG \longrightarrow B\mathbf{Tors}_t(G).
\end{equation}

\begin{lem}
\label{lem:BG}
	Let $\mathbf C$ be a small category, $t$ a Grothendieck topology on $\mathbf C$, and $G$ a $t$--sheaf of groups on $\mathbf C$.
	Then:
	\begin{enumerate}
		\item[(i)] The map \eqref{eqn:BG} is a weak equivalence of simplicial presheaves.
		\item[(ii)] There is a natural isomorphism
		\[
		\pi_0(B_tG)(-) \cong H^1_t(-,G).
		\]
		\item[(iii)] There is a canonical weak equivalence $\mathbf R\Omega B_tG\simeq G$.
	\end{enumerate}
\end{lem}

\begin{proof}
	It is clear that the map \eqref{eqn:BG} induces an isomorphism on $t$--sheaves of homotopy groups, so that it is a weak equivalence in the Jardine model structure. To deduce that it is a weak equivalence, it therefore suffices to show that the source and target are fibrant in the Jardine model structure.
	By Dugger–Hollander–Isaksen \cite[Corollary A.8]{DHI}, it suffices to show that, for every $U\in\mathbf{C}$, the simplicial sets $B_tG(U)$ and $B\mathbf{Tors}_t(G)(U)$ have no homotopy in dimensions $\geq 2$.  This statement is clear for the latter as it is the nerve of a groupoid.  To treat the former case, we recall a fact from simplicial homotopy theory: if $X$ is a simplicial set, then $X$ has no homotopy in dimensions $\geq k$ if and only if the homotopy fibers of the diagonal map $X \to X \times^h X$ have no homotopy in dimensions $\geq k-1$; this can be checked by assuming $X$ is a Kan complex and studying homotopy groups.  Thus, a simplicial set $X$ has no homotopy in dimensions $\geq 2$ if and only if its 3-fold diagonal
\[
X \longrightarrow  X\times^h_{X\times^h_{X\times^h X} X}X
\]
is a weak equivalence.  Since $R_t$ preserves homotopy pullbacks, it also preserves the property of having no homotopy in dimensions $\geq 2$.
	 This proves (i). Assertions (ii) and (iii) are true essentially by definition if we replace $B_tG$ by $B\mathbf{Tors}_t(G)$, so they both follow from (i).
\end{proof}

\subsubsection*{Torsors under $S$--group schemes}
Our main interest is to representability results for torsors under group schemes, so we now discuss that situation in greater detail.  Let $G$ be an $S$--group scheme and let $X$ be an $S$--scheme. By a \emph{$G$--torsor over $X$} we will mean a $G$--torsor in the sense of Definition~\ref{def:torsor}, for $\mathbf C$ the category of $S$--schemes and $t$ the fppf topology.  In the sequel $G$ will always be affine over $S$, and in that case a $G$--torsor over $X$ is automatically representable by an $S$--scheme, by Milne \cite[Theorem III.4.3 (a)]{Milne} (note: the implicit Noetherian hypothesis in Milne's argument is unnecessary).

If moreover $X$ and $G$ belong to $\Sm_S$, then taking $\mathbf C$ to be the category $\Sm_S$ with $t$ the \'etale topology one obtains an equivalent notion of torsor. Indeed, if $\pi\colon\mathscr P\to X$ is a $G$--torsor over $X$, then $\pi$ is finitely presented and smooth by the following lemma.  Since smooth morphisms admit sections \'etale locally, $\pi$ itself is a cover of $X$ in the \'etale topology which trivializes the torsor.

\begin{lem}
\label{lem:finitepresentation}
Suppose $G$ is an affine $S$--group scheme, $X$ is an $S$--scheme, and $\pi\colon\mathscr P\to X$ is a $G$--torsor over $X$. If $G\to S$ is finitely presented, flat, or smooth, then so is $\pi\colon \mathscr P\to X$.
\end{lem}

\begin{proof}
By definition, there exists an fppf cover $\{U_i\to X\}_{i\in I}$ such that $\mathscr P\times_XU_i\to U_i$ is isomorphic to $G\times_S U_i\to U_i$, which is finitely presented, flat, or smooth.  We conclude using the fact that each of these properties of a morphism is fppf-local on the target, by \cite[\href{http://stacks.math.columbia.edu/tag/02L0}{Tag 02L0 Lemma 34.19.11}, \href{http://stacks.math.columbia.edu/tag/02L2}{Tag 02L2 Lemma 34.19.13}, and \href{http://stacks.math.columbia.edu/tag/02VL}{Tag 02VL Lemma 34.19.25}]{stacks-project}.
\end{proof}

\begin{ex}
\label{ex:bundlesandtorsors}
Let $t$ be a topology on $\Sm_S$ in between the Zariski topology and the \'etale topology and let $n\geq 1$.  The groupoid of $GL_n$--torsors over a scheme is canonically equivalent to the groupoid of rank $n$ vector bundles. Since $GL_n$ is a smooth special group, any $GL_n$--torsor is $t$--locally trivial. In particular, by Lemma~\ref{lem:BG} (ii), we have
\[
\pi_0(B_tGL_n)(X)\cong \mathscr V_n(X)
\]
for any $X\in\Sm_S$, where $\mathscr V_n(X)$ denotes the set of isomorphism classes of rank $n$ vector bundles on $X$. Similarly, we have
\[
\pi_0(B_tSL_n)(X)\cong \mathscr V_n^o(X)\quad\text{and}\quad\pi_0(B_tSp_{2n})\cong \mathscr{HV}_{2n}(X),
\]
where $\mathscr V_n^o(X)$ (resp.\ $\mathscr{HV}_{2n}(X)$) is the set of isomorphism classes of rank $n$ oriented (resp.\ rank $2n$ symplectic) vector bundles (see the beginning of Section~\ref{ss:homotopyinvariance} for reminders about oriented and symplectic vector bundles).
\end{ex}

\subsubsection*{Affine representability for Nisnevich locally trivial $G$--torsors}
\begin{thm}
\label{thm:Gtorsors}
Suppose $G$ is a finitely presented smooth $S$--group scheme.  If $H^1_{\Nis}(-,G)$ is $\aone$--invariant on $\Sm_S^\aff$, then
\begin{enumerate}
	\item[(i)] The simplicial presheaf $R_{\Zar}\Singaone B_{\Nis}G$ is Nisnevich-local and $\A^1$--invariant.
	\item[(ii)] For every affine $X\in\Sm_S^\aff$, the canonical map
	\[
	H^1_{\Nis}(X,G) \longrightarrow [X,BG]_{\aone}
	\]
	is a bijection that is functorial with respect to $X$.
\end{enumerate}
\end{thm}

\begin{proof}
Since $B_{\Nis}G$ is Nisnevich-local by definition, it satisfies Nisnevich excision by \cite[Theorem 3.2.5]{PartI}.  Taking into account the identification $\pi_0(B_{\Nis}G)\cong H^1_{\Nis}(-,G)$ from point (ii) of Lemma~\ref{lem:BG}, we can apply \cite[Theorem~5.1.3]{PartI} to $B_{\Nis}G$, which implies (i) and (ii) (note also that $[X,B_{\Nis}G]_{\A^1}\cong [X,BG]_{\A^1}$ since $BG\to B_{\Nis}G$ is a Nisnevich-local equivalence).
 \end{proof}

\subsection{Application to homogeneous spaces}
\label{ss:applicationhomogeneousspaces}
Let $\mathbf C$ be a small category equipped with a Grothendieck topology $t$. Let $G$ and $H$ be $t$--sheaves of groups on $\mathbf C$ with $H\subset G$. We then have a homotopy fiber sequence of simplicial presheaves
\[
G/H \longrightarrow BH \longrightarrow BG,
\]
where $G/H$ denotes the presheaf $U\mapsto G(U)/H(U)$.
Applying the $t$--localization functor $R_t$, we obtain a homotopy fiber sequence of $t$--local simplicial presheaves
\begin{equation}
\label{eqn:homogeneousSpaceAsFiber}
a_t(G/H) \longrightarrow B_t H \longrightarrow B_t G.
\end{equation}
We now restrict attention to $\mathbf{C} = \Sm_S$ with the goal of applying Theorem~\ref{thm:fiberaffinerepresentability}.  For geometric applications, we need to better understand the sheaf $a_t(G/H)$.

\subsubsection*{Homogeneous spaces: topologies and quotient sheaves}
Write $rX$ for the presheaf on the category of $S$--schemes represented by an $S$--scheme $X$, and $r'X$ for the restriction of the presheaf $rX$ to $\Sm_S$. Suppose that $G$ and $H$ are finitely presented smooth $S$--group schemes, and that $H$ is a closed subgroup of $G$.  The right translation action of $H$ on $G$ is scheme-theoretically free and it follows from a result of Artin \cite[Corollary 6.3]{Artin} that the sheaf $a_{\fppf}(rG/rH)$ is representable by an $S$--algebraic space.  Two questions naturally present themselves: first, when does the fppf sheaf quotient coincide with the Zariski or Nisnevich sheaf quotient and second, is the fppf-sheaf $a_{\fppf}(rG/rH)$ representable by a smooth scheme?  We address the first question here; we answer the second question in various cases in Section~\ref{ss:grouptorsorshomogeneousspaces}.

\begin{lem}
\label{lem:zariskiquotient}
Suppose $G$ is a finitely presented $S$--group scheme and $H \subset G$ is a finitely presented closed $S$--subgroup scheme. Assume that $H$ is flat over $S$ and that the quotient $G/H$ exists as an $S$--scheme.  Then $G\to G/H$ is an $H$--torsor, and the following statements hold.
\begin{enumerate}
	\item[(i)] If $t$ is a subcanonical topology on $S$--schemes such that the map $G \to G/H$ is $t$--locally split, then $r(G/H) \cong a_t(rG/rH)$.
	\item[(ii)] If $G$ is smooth over $S$, then $G/H$ is smooth over $S$. Moreover, if $t$ is a subcanonical topology on $\Sm_S$ such that the map $G\to G/H$ is $t$--locally split, then $r'(G/H)\cong a_t(r'G/r'H)$.
\end{enumerate}
\end{lem}

\begin{proof}
	By a theorem of Anantharaman \cite[Appendice I, Th\'eor\`eme 6]{Anantharaman}, we have $r(G/H)\cong a_{\fppf}(rG/rH)$. In particular, $G\to G/H$ is an $H$--torsor, and hence it is flat by Lemma~\ref{lem:finitepresentation}. If $G$ is smooth, it follows from \cite[Proposition 17.7.7]{EGA44} that $G/H$ is also smooth. If $G\to G/H$ is $t$--locally split, then $rG\to r(G/H)$ is an epimorphism of $t$--sheaves. By \cite[Proposition 4.3 (2)]{SGA41}, this implies that $r(G/H)$ is the coequalizer of the equivalence relation $rG\times_{r(G/H)}rG\cong rG\times rH\rightrightarrows rG$ in the category of $t$--sheaves, which exactly means that $r(G/H)\cong a_t(rG/rH)$. The second statement is proved in the same way.
\end{proof}

\subsubsection*{Affine representability for homogeneous spaces}
\begin{thm}
\label{thm:homogeneousrep}
Suppose $G$ is a finitely presented smooth $S$--group scheme and $H \subset G$ is a finitely presented smooth closed $S$--subgroup scheme such that the quotient $G/H$ exists as an $S$--scheme. Suppose that $G\to G/H$ is Nisnevich locally split and that $H^1_{\Nis}(-,G)$ and $H^1_{\Nis}(-,H)$ are $\A^1$--invariant on $\Sm_S^\aff$. Then $G/H$ is $\A^1$--naive. In particular, for every $X\in\Sm_S^\aff$, there is a bijection
\[
\pi_0(\Singaone G/H)(X) \cong [X,G/H]_{\A^1}.
\]
\end{thm}

\begin{proof}
The assumption on $G \to G/H$ combined with Lemma~\ref{lem:zariskiquotient} allow us to conclude that $r'(G/H) \cong a_{\Nis}(r'G/r'H)$ and thus the homotopy fiber sequence \eqref{eqn:homogeneousSpaceAsFiber} has the form $r'(G/H) \to B_{\Nis} H \to B_{\Nis} G$.  The simplicial presheaves $B_{\Nis}G$ and $B_{\Nis}H$ are Nisnevich-local and hence satisfy Nisnevich excision by \cite[Theorem 3.2.5]{PartI}.  The result is now a direct application of Theorem~\ref{thm:fiberaffinerepresentability}, taking into account Lemma~\ref{lem:BG} (ii).
\end{proof}

\section{Homotopy invariance for torsors under group schemes}
\label{sec:groupstorsors}
The main goal of this section is to study $\aone$--invariance of the functors $H^1_{\Nis}(-,G)$ for $G$ a linear group.  Section~\ref{ss:grouptorsorshomogeneousspaces} reviews basic definitions about group schemes, torsors and homogeneous spaces; it also collects a number of results that will be used later in the text.  Section~\ref{ss:localtoglobal} establishes an analog of the local-to-global principle (a.k.a. ``Quillen patching'') for torsors under linear group schemes under rather general hypotheses; the main result is Theorem~\ref{thm:localtoglobalprinciple}.  Finally, Section~\ref{ss:homotopyinvariance} proves general homotopy invariance results; the main results are Theorems~\ref{thm:sympletichomotopyinvariance} and~\ref{thm:hoinvisotropic}.  For simplicity, we assume throughout this section that the base scheme $S$ is the spectrum of a commutative ring $R$.  In general there is a tradeoff between generality of the group $G$ under consideration and the base ring $R$.

\subsection{Reductive group schemes and homogeneous spaces: recollections}
\label{ss:grouptorsorshomogeneousspaces}
The goal of this section is to recall some basic definitions and properties of group schemes, torsors and homogeneous spaces over rather general bases.  Rather than attempting to be exhaustive, we only aim to point the reader to places in the literature where they can find the required results.  The grouping of these results is slightly eclectic: only a very small portion of the definitions and results established here will be used in the remainder of Section~\ref{sec:groupstorsors}.  Many of the results we state here are significantly easier to establish (or even unnecessary) if the base ring $R$ is a field.

\subsubsection*{Linear and reductive group schemes}
We write $GL_{n,R}$ for the general linear group scheme over $R$ and $\gm{}_{,R}$ for $GL_{1,R}$.  If $R$ is clear from context, we will drop it from the notation.

\begin{defn}
\label{defn:lineargroupscheme}
By a {\em linear $R$--group scheme}, we mean a group scheme $G$ over $R$ admitting a finitely presented closed immersion group homomorphism $G \to GL_{n,R}$.
\end{defn}

Later, the homotopy invariance results we establish will require much more stringent hypotheses on $G$.  We use the definition of {\em reductive} (resp. {\em semi-simple}) $R$--group scheme of Demazure--Grothendieck \cite[Expos{\'e} XIX Definition 2.7]{SGA3.3}: a reductive (resp. {\em semi-simple}) $R$--group scheme is a smooth, affine $R$--group scheme with geometric fibers that are connected reductive (resp. {\em semi-simple}) groups in the usual sense \cite[Expos{\'e} XIX 1.6]{SGA3.3}, i.e., have trivial unipotent radical (resp. radical).  Recall that a reductive $R$--group scheme $G$ is called {\em split} if it contains a split maximal torus \cite[Expos\'e XXII D\'efinition 1.13]{SGA3.3}.  Any split reductive group scheme is pulled back from a unique ``Chevalley'' group scheme over $\Spec \Z$.

If $R$ is a field, it is a well-known consequence of the classification of reductive groups that reductive $R$--group schemes are linear $R$--group schemes.  If $R$ is no longer a field, the connection between ``reductive'' and ``linear'' becomes more complicated, as the following example demonstrates.

\begin{ex}
Groups of multiplicative type need not be linear in general \cite[Expose IX D\'efinition 1.1]{SGA3.2}.   Indeed, \cite[Expos\'e XI Remarque 4.6]{SGA3.2} explains that if $R$ is a Noetherian and connected ring, then a group $G$ of multiplicative type admits an embedding in $GL_n$ if and only if it is isotrivial.
\end{ex}

Nevertheless, the following result shows that, assuming suitable hypotheses on the base, reductive $R$--group schemes are always linear.

\begin{prop}[Thomason]
\label{prop:criteriaforlinearity}
Suppose $G$ is a reductive $R$--group scheme.  Assume one of the following additional hypotheses holds:
\begin{enumerate}
\item[(i)] $R$ is regular and Noetherian; or
\item[(ii)] $G$ is split.
\end{enumerate}
Then $G$ is a linear $R$--group scheme.
\end{prop}

\begin{proof}
If $G$ is split, we can assume that $R=\Z$ and in particular that $R$ is regular Noetherian. In that case, the result follows from Thomason \cite[Corollary 3.2 (3)]{ThomasonResolution}.
\end{proof}

\begin{rem}
Thomason actually gives a sufficient condition for a group scheme to admit a closed immersion group homomorphism into the automorphism group scheme of a vector bundle over an arbitrary base $S$ \cite[Theorem 3.1]{ThomasonResolution}.  Since we have in mind applications to homotopy invariance, we have restricted attention to spectra of regular rings.
\end{rem}

\subsubsection*{Homogeneous spaces for reductive groups}
Suppose $G$ is a reductive $R$--group scheme and $\lambda: \gm{} \to G$ is a homomorphism of $R$--group schemes. Via $\lambda$, we may consider the $\gm{}$--action $\lambda: \gm{} \times G \to G$ defined pointwise by the formula $\lambda(t,g) := \lambda(t) g \lambda(t)^{-1}$.  We can define a subfunctor $P_G(\lambda) \subset G$ consisting of those points $g \in G$ such that $\lim_{t \to 0} \lambda(t,g)$ exists and a sub-functor $U_G(\lambda) \subset G$ consisting of those points $g \in G$ such that $\lim_{t\to 0} \lambda(t,g) = 1$ (see Conrad \cite[Theorem 4.1.7]{Conrad} for precise definitions).  By \cite[Theorem 4.1.7]{Conrad} both of these functors are representable by $R$--subgroup schemes of $G$; since we assumed $G$ reductive it follows also that $P_G(\lambda)$ and $U_G(\lambda)$ are smooth and connected.  By \cite[Example 5.2.2]{Conrad} $P_G(\lambda)$ is parabolic, and $U_G(\lambda)$ is a closed normal $R$--subgroup scheme whose geometric fibers correspond to unipotent radicals of the geometric fibers of $P_G(\lambda)$ \cite[Corollary 5.2.5]{Conrad}; we will abuse terminology and refer to $U_G(\lambda)$ as the unipotent radical of $P_G(\lambda)$.

If $Z_G(\lambda)$ is the centralizer of $\lambda$, then by \cite[Definition 5.4.2]{Conrad} and the subsequent discussion, $Z_G(\lambda)$ is a Levi factor of $P_G(\lambda)$, i.e., $Z_G(\lambda)$ is a smooth reductive $R$--group scheme, and there is a semi-direct product decomposition of the form $Z_G(\lambda) \ltimes U_G(\lambda) \cong P_G(\lambda)$.  This description of parabolics, their unipotent radicals and Levi factors is called a ``dynamic'' description in \cite{ConradGabberPrasad, Conrad} (since it arises from a study of ``flows'' under an action of $\gm{}$).  We use these ideas to establish the following result.

\begin{lem}
\label{lem:parabolics}
Suppose $R$ is a connected ring, $G$ is a reductive $R$--group scheme, $P \subset G$ is a parabolic $R$--subgroup scheme and $L$ is a Levi factor of $P$.  The following statements hold.
\begin{enumerate}
\item[(i)] The quotients $G/L$ and $G/P$ exist as smooth $R$--schemes.
\item[(ii)] The morphism $G \to G/L$ is a generically trivial $L$--torsor.
\item[(iii)] The morphism $G/L \to G/P$ is a composition of torsors under vector bundles.
\end{enumerate}
\end{lem}

\begin{proof}
For later use, we observe that since $R$ is assumed connected and $L$ is presumed to exist, by Gille \cite[Th\'eor\`eme 9.3.1]{GilleLuminy}, there is a cocharacter $\lambda: \gm{} \to G$ such that $P = P_G(\lambda)$ and $L = Z_G(\lambda)$.  If $S$ is the spectrum of a field, which is the case we will use later, the fact that all pairs $(P,L)$ consisting of a parabolic together with a Levi factor, are of the form $(P_G(\lambda),Z_G(\lambda))$ for a suitable cocharacter $\lambda$ is contained in \cite[Proposition 2.2.9]{ConradGabberPrasad}.

For Point (i), begin by observing that since $P$ is a parabolic subgroup of $G$ it is a self-normalizing subgroup \cite[Corollary 5.2.8]{Conrad}.  The quotients $G/L$ and $G/P$ exist as smooth $R$--schemes by \cite[Theorems 2.3.1 and 2.3.6]{Conrad} (and, by Lemma~\ref{lem:zariskiquotient}, the morphisms $G \to G/L$ and $G \to G/P$ are an $L$--torsor and a $P$--torsor, respectively).

For Point (ii), set $U^- = U_G(-\lambda)$, i.e., the ``unipotent radical'' of an opposite parabolic.  We know that there is a dense open subscheme of $G$ isomorphic to $U^- \times P$ \cite[Theorem 4.1.7]{Conrad} (here and below, we will refer to this as the ``big cell'').  The image of this open subscheme in $G/L$, which is isomorphic to $U^- \times P/L$, is again open and dense since $G \to G/L$ is smooth and surjective.  The Levi decomposition yields an isomorphism of schemes $P \cong L \times U$, and thus an identification $P/L \cong U$.  Under these identifications, the unit map $U \to P$ provides a morphism $U^{-} \times U \to U^{-} \times L \times U$, which yields the required generic trivialization.

For Point (iii), let $U$ be the unique smooth closed normal $R$--subgroup scheme of $P$ whose geometric fibers coincide with the unipotent radicals of the geometric fibers of $P$, which is guaranteed to exist by \cite[Corollary 5.2.5]{Conrad}.  By the uniqueness assertion, $U \cong U_G(\lambda)$ for the character whose existence we observed in the first paragraph.  By \cite[Theorem 5.4.3]{Conrad}, $U$ admits a finite descending filtration by $Aut_{P/R}$--stable closed normal smooth $R$--subgroup schemes $U_i$ with successive subquotients $U_i/U_{i+1}$ isomorphic to $P$--equivariant vector bundles over $R$.  Moreover, the isomorphism $P/L \cong U$ described in Point (ii) is actually $P$--equivariant.

Now, the morphism $G/L \longrightarrow G/P$ is $G$--equivariant by definition.  The scheme-theoretic fiber over the identity coset in $G/P$ is isomorphic to the quotient $P/L$ and there is an induced $G$--equivariant isomorphism $G \times^P P/L \isomt G/L$ under which the morphism $G/L \to G/P$ is sent to the projection onto the first factor.  In particular, since $P/L \cong U$ is smooth, $G \times P/L \to G$ is smooth and since smoothness is fppf local on the base \cite[\href{http://stacks.math.columbia.edu/tag/02VL}{Tag 02VL Lemma 34.19.25}]{stacks-project}, we conclude that $G/L \to G/P$ is also smooth.  By discussion of the previous paragraph, the morphism $G/L \to G/P$ thus factors successively through morphisms of the form
\begin{equation}
\label{eqn:towerofvectorbundletorsors}
G \times^P U/U_{i+1} \longrightarrow G \times^P U/U_{i}.
\end{equation}
To finish the proof, it suffices to inductively establish that each morphism in (\ref{eqn:towerofvectorbundletorsors}) is a torsor under a vector bundle.

Each morphism $U/U_{i+1} \to U/U_i$ is, by construction, a torsor under the vector bundle $U_i/U_{i+1}$ and, as we observed above, provided with a $P$--equivariant structure.  If $\mathscr{E}$ is a quasi-coherent sheaf on a scheme $X$, then $H^1_{\fppf}(X,\mathscr{E}) = H^1_{\Zar}(X,\mathscr{E})$ by \cite[\href{http://stacks.math.columbia.edu/tag/03DR}{Tag 03DR Proposition  34.7.10}]{stacks-project}.  Since $H^1_{\fppf}(X,\mathscr{E})$ parameterizes fppf-torsors under the quasi-coherent sheaf $\mathscr{E}$, the $P$--equivariant structure on $U_i/U_{i+1}$ allows us to conclude, by fppf-descent, that $G \times^P U_i/U_{i+1}$ is a torsor under a vector bundle on $G/P$.  In other words, each morphism in (\ref{eqn:towerofvectorbundletorsors}) is again a torsor under the vector bundle $U_i/U_{i+1}$.
\end{proof}

\begin{rem}
\label{rem:levisexist}
A number of remarks are in order.
\begin{enumerate}
\item Since $R$ a connected ring, it is not necessary to assume in the statement above that $L$ exists; this follows from Conrad \cite[Corollary 5.4.8]{Conrad}.  If we were to work over a non-affine base scheme, parabolics need not have Levi factors (see \cite[Example 5.4.9]{Conrad} for more details).  By reorganizing the proof, the argument presented in Point (iii) actually shows that the quotient $G/L$ exists assuming we know $G/P$ to exist and the relevant results on the structure of $U$.
\item  By Lemma~\ref{lem:finitepresentation}, since $L$ is a smooth $R$--group scheme by assumption, $G \to G/L$ is \'etale locally trivial.  If $R$ is Noetherian and regular, then the morphism $G \to G/L$ being generically trivial is tantamount to $G \to G/L$ being Nisnevich locally trivial.  To prove this, it suffices to show that generically trivial $L$--torsors over Henselian local rings are trivial.  If $G$ is split reductive, then $L$ is as well, and the asserted triviality follows from Bia\l{}ynicki-Birula \cite[Proposition 2]{BialynickiBirula}. If $G$ is not necessarily split, then $L$ can be an arbitrary reductive group and one can appeal to Nisnevich \cite[Th\'eor\`eme 4.5]{Nisnevich} to deduce the required triviality result (Nisnevich makes a statement for semi-simple group schemes, but it is true more generally, see Fedorov and Panin \cite[\S 1.1]{FedorovPaninGS}).
\item If $G$ is split, it is possible to use translation of the big cell by elements of the Weyl group to produce an explicit Zariski local trivialization of $G \to G/L$.  In fact, even if $G$ is not split, to establish Zariski local triviality of $G \to G/L$ (or, equivalently, $G \to G/P$), it suffices to know that the $G(R)$--translates of the big-cell form an open cover of $G/L$ (or $G/P$).  If $R$ is an infinite field, this kind of result follows from the fact that the image of $G(R)$ in $G/P(R)$ is Zariski dense (via the unirationality of $G$).
\item In contrast, if $R$ is a finite field (and $G$ is non-split), it is {\em a priori} not obvious that $G(R)$ translates of the big cell cover $G/L$ (or $G/P$).  Nevertheless, assuming the Grothendick--Serre conjecture, one knows that $G \to G/L$ is Zariski locally trivial.  If $R$ is the spectrum of a finite field, the Grothendieck--Serre conjecture was established by Gabber for reductive groups coming from the ground field (unpublished), but another proof of a more general case was recently given by Panin \cite{PaninGSFinite} (see also \cite{FedorovPaninGS}).
\end{enumerate}
\end{rem}

Write $SO_n$ for the split special orthogonal group over $R$.  We restrict attention to the case where $2$ is a unit in $R$ so we can view $SO_n$ as the $R$--subgroup scheme of $GL_{n}$ consisting of automorphisms of the standard hyperbolic form $q_n$ with trivial determinant (see, e.g., Conrad \cite[Definition C.1.2]{Conrad}); for more details on special orthogonal groups, see \cite[Appendix C]{Conrad}).

\begin{lem}
\label{lem:orthogonalquotient}
If $R$ is a ring in which $2$ is invertible, then the following statements hold.
\begin{enumerate}
\item[(i)] If $n \geq 3$, the quotient $SO_n/SO_{n-1}$ exists and is isomorphic to a quadric hypersurface in ${\mathbb A}^n_R$ defined by the equation $q_n = 1$.
\item[(ii)] If $n \geq 3$, the projection morphism $SO_n \to SO_n/SO_{n-1}$ makes $SO_n$ into a Zariski locally trivial $SO_{n-1}$--torsor over the quotient.
\end{enumerate}
\end{lem}

\begin{proof}
Without loss of generality, we can take $R = \Z[1/2]$, which is Noetherian of dimension $\leq 1$.  Since $SO_{n-1}$ is a closed $R$--subgroup scheme of $SO_n$, the quotient $SO_n/SO_{n-1}$ exists as a scheme \cite[Th\'eor\`eme 4.C]{Anantharaman}.

To identify this quotient with the quadric in the statement, we proceed as follows.  Since $SO_{n-1}=SO_n\cap SL_{n-1}$ inside of $SL_n$, the inclusion $SO_n\subset SL_n$ induces a monomorphism $SO_n/SO_{n-1} \hookrightarrow SL_n/SL_{n-1}$.  Note that if $A$ is an $R$--algebra, the map sending $X \in SL_n(A)$ to its first row and the first column of its inverse determines an isomorphism $SL_n/SL_{n-1} \cong \Spec R[x_1,\ldots,x_{2n}]/(q_{2n} - 1)$.  If we restrict $X \in SO_n(A)$ and if $J$ is the symmetric matrix corresponding to the bilinear form associated with $q_n$, then the orthogonality condition imposes the relation $X^{-1} = JX^T$.  Using this observation, it is straightforward to check that the image is isomorphic, in suitable coordinates, to a sub-quadric given by the equation $q_n = 1$.

For the second statement, observe that morphisms $X\to SO_n/SO_{n-1}$ classify $SO_{n-1}$--torsors which are trivial after stabilization to $SO_n$--torsors. The Witt cancellation theorem (see Milnor and Husemoller \cite[Lemma 6.3]{MilnorHusemoller}) implies that, over a local ring in which $2$ is invertible, such an $SO_{n-1}$--torsor is already trivial.
\end{proof}

\subsection{The local-to-global principle for torsors under linear group schemes}
\label{ss:localtoglobal}
In this section we establish a local-to-global principle or ``Quillen patching'' for torsors under linear $R$--group schemes in the sense of Definition~\ref{defn:lineargroupscheme}.  The main result of this section is Theorem~\ref{thm:localtoglobalprinciple}, which is a multi-variable analog of a result of Quillen \cite[Theorem 1]{Quillen} along the lines of Lam \cite[Theorem V.1.6]{Lam}.  As will be clear from the presentation, the argument follows quite closely that for projective modules given in \cite[Chapter V.1]{Lam}.

That the local-to-global principle holds for torsors under linear group schemes is certainly ``well-known to experts'', under suitable hypotheses.  For example, Raghunathan \cite{Raghunathan78} states (without proof) that Quillen's local-to-global principle holds for linear algebraic groups over a field and Bass--Connell--Wright developed an axiomatic method to establish such results  \cite[Proposition 3.1]{BCW}; in particular, the latter approach applies for various classical groups \cite[Remark 4.15.4]{BCW} over a general base ring.  Nevertheless, since we could not find a suitable published reference for precisely what we needed, in the interest of completeness, we decided to collect the necessary results here.

\subsubsection*{Modifying automorphisms}
We begin by generalizing \cite[Lemma 1]{Quillen} (also \cite[Corollary V.1.2]{Lam}) and  \cite[Corollary V.1.3]{Lam} to linear $R$--group schemes over an arbitrary commutative ring $R$.  The following pair of results are due to Moser \cite[Lemmas 3.5.3--3.5.5]{Moser} (though our hypotheses differ slightly); we include them here for the convenience of the reader.

\begin{lem}
\label{lem:modifyingautomorphismsgtorsors}
Let $R$ be a commutative ring, let $G$ be a linear $R$--group scheme, let $f \in R$, and let $\theta(t) \in G(R_f[t])$ be such that $\theta(0) = 1 \in G(R_f)$.  There exists an integer $s \geq 0$ such that for any $a,b \in R$ with $a-b \in f^s R$, there exists $\psi \in G(R[t])$ with $\psi(0) = 1$ and such that $\psi_f(t) = \theta(at)\theta(bt)^{-1} \in G(R_f[t])$.
\end{lem}

\begin{proof}
Since $G$ is a linear $R$--group scheme, by definition there is a finitely presented closed immersion $G \to GL_{n}$.  For $s \in {\mathbb N}$, set $\psi_s(t,x,y) := \theta((x + f^sy)t)\theta(xt)^{-1} \in G(R_f[t,x,y])$.  It suffices to show that there exists $s$ such that $\psi_s$ can be lifted to an element $\psi_s \in G(R[t,x,y])$.  Indeed, in that case, by specializing with $x = b$, $a = b + f^s\alpha$, we see that $\theta(at)\theta(bt)^{-1} = \psi_s(t,b,\alpha)$ lifts as well.  By the proof of \cite[Lemma 1]{Quillen}, we know that there exists $s$ such that $\psi_s(t,x,y)$ lifts to an element of $GL_n(R[t,x,y])$ and such that $\psi_s(0,x,y) = 1$ (see also \cite[Theorem V.1.1]{Lam}).  Observe that, by definition, $\psi_s(t,x,0) = 1$ and thus $\psi_s(t,x,0) \in G(R[x,t])$.

It remains to show that there exists $i\geq 0$ such that $\psi_s(t,x,f^iy) \in G(R[t,x,y])$.  We first recast this in ring-theoretic terms.  Set $A := R[t,x]$, let $B$ be the coordinate ring of $GL_n$, and let $I \subset B$ be the finitely generated ideal defining $G$.  The lift of $\psi_s$ is given by a homomorphism $\varphi: B \to A[y]$, and we want to show that, for some $i\geq 0$, $\varphi(-)(f^iy)$ vanishes on $I$.  We claim that, for every $r\in I$, there exists an integer $i_r$ such that  $\varphi(r)(f^i y) = 0$ for $i \geq i_r$.  If $J\subset I$ is a finite generating set and $i = \max_{r\in J}{i_r}$, then $i$ will have the desired property.

Note that $\varphi$ has the following properties: if $ev_0: A[y] \to A$ is the evaluation homomorphism, then the composites $ev_0 \circ \varphi: B \to A$ and $B \to A[y] \to A_f[y]$ both vanish on $I$.  If $r\in I$ and $P:=\phi(r)\in A[y]$, these properties imply that $P = y Q$ for some $Q \in A[y]$ and that $f^{i_r}P=0$ for some $i_r\geq 0$.  Combining these two observations, we have $0=f^{i_r}P = f^{i_r}yQ$.  Therefore, $f^{i_r}Q = 0$ as well.  Thus, $P(f^i y) =  f^{i}yQ(f^iy) = 0$ for all $i \geq i_r$, which is what we wanted to show.
\end{proof}

\begin{lem}
\label{lem:modifyautomorphisms}
Let $R$ be a commutative ring and $G$ a linear $R$--group scheme.  Given $f_0,f_1 \in R$ such that $f_0R + f_1R = R$, and $\theta \in G(R_{f_0f_1}[t])$ with $\theta(0) = 1$, then we can find $\tau_i \in G(R_{f_i}[t])$ with $\tau_i(0) = 1$ such that $\theta = \tau_0 \tau_{1}^{-1}$.
\end{lem}

\begin{proof}
Let $\theta(t) \in G(R_{f_0f_1}[t])$.  We can apply Lemma~\ref{lem:modifyingautomorphismsgtorsors} to the localizations $R_{f_0} \to R_{f_0f_1}$ and $R_{f_1} \to R_{f_1f_0}$: pick an integer $s$ that suffices for both localizations.  For any $b \in R$, we can write
\[
\theta(t) = [\theta(t)\theta(bt)^{-1}] \theta(bt).
\]
If $f_0R + f_1R = R$, then the same thing is true for $f_0^s$ and $f_1^s$.  Thus, we can pick $b \in f_1^s R$ such that $1-b \in f_0^s R$. In that case, $\theta(t)\theta(bt)^{-1} \in G(R_{f_1}[t])_{f_0}$ and $\theta(bt) \in G(R_{f_0}[t])_{f_1}$ lift to elements $\tau_1$ and $\tau_0$ with the stated properties.
\end{proof}

\begin{rem}
Lemma~\ref{lem:modifyingautomorphismsgtorsors} implies ``Axiom Q'' (in the sense of Bass, Connell, and Wright \cite[\S 1.1]{BCW}) holds for the functor on $R$--algebras determined by $G$.  Lemma~\ref{lem:modifyautomorphisms} essentially corresponds to \cite[Theorem 2.4]{BCW}.
\end{rem}

\subsubsection*{The local-to-global principle}
Let $R$ be a commutative ring and suppose $G$ is a linear $R$--group scheme.  If $A$ is a commutative $R$--algebra, by a $G$--torsor over $A$ we will mean a $G$--torsor over $\Spec A$; by assumption our $G$--torsors are locally trivial in the fppf-topology (see Definition~\ref{def:torsor} and the discussion just prior to Lemma~\ref{lem:finitepresentation} for more details).  A $G$--torsor over $A[t_1,\ldots,t_n]$ that is pulled back from a $G$--torsor over $A$ will be called {\em extended from $A$}.  For the remainder of this section, we will essentially confine our attention to a {\em fixed} $G$--torsor ${\mathscr P}$, which will be important for subsequent applications.

\begin{prop}
\label{prop:quillenpatching1var}
Let $R$ be a commutative ring.  If $\mathscr{P}$ is a $G$--torsor over $R[t]$, then the set $Q({\mathscr P})$ consisting of $g \in R$ such that ${\mathscr P}|_{\Spec R_g[t]}$ is extended from $R_g$ is an ideal in $R$.
\end{prop}

\begin{proof}
It is immediate that $Q(\mathscr{P})$ is closed under multiplication by elements in $R$.  Thus, we have to show that if $f_0,f_1 \in Q(\mathscr{P})$, then $f = f_0 + f_1$ lies in $Q(\mathscr{P})$ as well.  After replacing $R$ by $R_f$, we can assume that $f_0R + f_1R = R$.

Write $0: \Spec R \to \aone_R$, and $pr: \aone_R \to \Spec R$ for the zero section and the structure morphism.  Thus, suppose $\mathscr{P}$ is a $G$--torsor over $R[t]$ and assume that the restrictions $\mathscr{P}_{i} := \mathscr{P}|_{\Spec R_{f_i}[t]}$ are extended.  We want to show that $\mathscr{P} \cong pr^*0^*\mathscr{P}$.

By assumption, there are isomorphisms $u_i:\mathscr P_{i}\cong pr^*0^*\mathscr P_i$ over $R_{f_i}[t]$. By modifying $u_i$ if necessary, we may assume that $0^*u_i = 1$.  Let $\mathscr P_{01}$ be the restriction of $\mathscr P$ to $R_{f_0f_1}[t]$. Then $u_0$ and $u_1$ restrict to give two isomorphisms $(u_0)_{f_1}, (u_1)_{f_0} : \mathscr P_{01}\cong pr^*0^*\mathscr P_{01}$.  If we set $\theta = (u_1)_{f_0}(u_0)_{f_1}^{-1} \in G(R_{f_0f_1}[t])$, then there is a commutative diagram of the form
\[
\xymatrix{
\mathscr{P}_{0} \ar[d]^{u_0}&&\ar[ll]  \mathscr{P}_{01} \ar[rr]\ar[dl]^{(u_0)_{f_1}}\ar[dr]^{(u_1)_{f_0}} &&  \mathscr{P}_{1}\ar[d]^{u_1} \\
pr^*0^*\mathscr{P}_0 & \ar[l] pr^*0^*\mathscr{P}_{01} \ar[rr]^{\theta}& & pr^*0^*\mathscr{P}_{01} \ar[r] & pr^*0^*\mathscr{P}_1.
}
\]
If $\theta$ is the identity, then by fppf descent for $G$--torsors, the isomorphisms $u_0$ and $u_1$ glue to give an isomorphism $\mathscr P\cong pr^*0^*\mathscr P$, as desired.  If not, since $0^*u_i = 1$, we see that $\theta$ restricts along $t = 0$ to the identity.  Then, Lemma~\ref{lem:modifyautomorphisms} guarantees that we can find $\tau_i \in G(R_{f_i}[t])$ such that $\tau_i(0) = 1$ and such that $\theta = \tau_0 \tau_1^{-1}$.  Thus, $(\tau_0u_0)_{f_1} = (\tau_1 u_1)_{f_0}$ and replacing $u_0$ by $\tau_0u_0$ and $u_1$ by $\tau_1 u_1$, we can glue these isomorphisms to conclude that $\mathscr{P}$ is extended.
\end{proof}

\begin{thm}[Local-to-global principle]
\label{thm:localtoglobalprinciple}
Let $R$ be a commutative ring and suppose $G$ is a linear $R$--group scheme.  If $\mathscr{P}$ is a $G$--torsor over $R[t_1,\ldots,t_n]$, then
\begin{itemize}
\item[$(A_n)$] the set $Q({\mathscr P})$ consisting of $g \in R$ such that ${\mathscr P}|_{\Spec R_g[t_1,\ldots,t_n]}$ is extended from $R_g$ is an ideal in $R$.
\item[$(B_n)$] If ${\mathscr P}|_{\Spec R_{{\mathfrak m}}[t_1,\ldots,t_n]}$ is extended for every maximal ideal ${\mathfrak m} \subset R$, then ${\mathscr P}$ is extended.
\end{itemize}
\end{thm}

\begin{proof}
We know that $(A_1)$ holds by Proposition~\ref{prop:quillenpatching1var}.\newline

\noindent {\em We show $(A_n) \Longrightarrow (B_n)$}.  It suffices to check that for $\mathscr{P}$ satisfying the conditions in $(B_n)$ that the ideal $Q({\mathscr P})$ is the unit ideal in $R$.  To this end, let $\mathscr{P}|_0$ be the pullback of $\mathscr{P}$ along the zero section $\Spec R \to \Spec R[t_1,\ldots,t_n]$ and let ${\mathscr P}'$ be the pullback of $\mathscr{P}|_0$ along the structure map $\Spec R[t_1,\ldots,t_n] \to \Spec R$.

For any maximal ideal ${\mathfrak m} \subset R$, since $\mathscr{P}|_{\Spec R_{\mathfrak m}[t_1,\ldots,t_n]}$ is by assumption extended, we know there is an isomorphism $\varphi: \mathscr{P}|_{\Spec R_{\mathfrak m}[t_1,\ldots,t_n]} \isomt \mathscr{P}'|_{\Spec R_{\mathfrak m}[t_1,\ldots,t_n]}$.  Since $G$--torsors over affine bases are of finite presentation under our hypotheses by Lemma~\ref{lem:finitepresentation}, there exists $g \in R \setminus {\mathfrak m}$ such that $\varphi$ is the localization of an isomorphism of torsors over $\Spec R_g[t_1,\ldots,t_n]$.  It follows that $g \in Q({\mathscr P})  \setminus {\mathfrak m}$ and therefore that $Q(\mathscr{P})$ is not contained in ${\mathfrak m}$, i.e., $Q({\mathscr P}) = R$.\newline

\noindent {\em We show $(A_1) \Longrightarrow (A_n)$}. We proceed by induction on $n$.  Assume therefore that $(A_{n-1})$ holds.  By the conclusion of the previous step, this means $(B_{n-1})$ holds as well. Form the set $Q(\mathscr{P})$ as in $(A_n)$.  It is straightforward to check that $R \cdot Q(\mathscr{P}) \subset Q({\mathscr P})$ and therefore it suffices to show that if $f_0,f_1 \in Q(\mathscr{P})$, then $f_0 + f_1 \in Q(\mathscr{P})$ as well.

Write $f = f_0 + f_1$.  Consider the quotient map $R[t_1,\ldots,t_n] \to R[t_1,\ldots,t_{n-1}]$ and set $\mathscr{P}|_{t_n=0}$ to be the restriction of $\mathscr{P}$ under the corresponding morphism of schemes.  Likewise, write ${\mathscr P}|_0$ for the restriction of $\mathscr{P}$ along the zero section as in the previous step.  Applying $(A_1)$ to the map $R[t_1,\ldots,t_{n-1}] \to R[t_1,\ldots,t_{n-1}][t_n]$, we conclude that $\mathscr{P}_f$ is extended from $(\mathscr{P}|_{t_n=0})_f$.

We claim that $(\mathscr{P}|_{t_n=0})_f$ is itself extended from $R_f$.  If that is the case, then $\mathscr{P}_f$ is extended and so $f \in Q(\mathscr{P})$.  Since $(B_{n-1})$ holds, it suffices to show that $(\mathscr{P}|_{t_n=0})_f$ is extended upon restriction to every maximal ideal ${\mathfrak m} \in R_f$.  Write ${\mathfrak m} = {\mathfrak p}_f$ where ${\mathfrak p}$ is the pre-image of ${\mathfrak m}$ under the localization map $R \to R_f$.  Since $f \notin {\mathfrak p}$ it follows that either $f_0$ or $f_1$ is not in ${\mathfrak p}$; without loss of generality, we can assume that $f_0 \notin {\mathfrak p}$.  By assumption, however, ${\mathscr P}_{f_0}$ is extended from $({\mathscr P}_0)_{f_0}$ so we conclude that the restriction of $(\mathscr{P}|_{t_n=0})_f$ to the maximal ideal ${\mathfrak m}$ is extended from $({\mathscr P}_0)_{\mathfrak p}$, which is what we wanted to show.
\end{proof}

\begin{cor}
\label{cor:localtoglobalisotropic}
Let $G$ be a reductive $R$--group scheme. If $R$ is regular Noetherian or $G$ is split, then the local-to-global principle holds for $G$--torsors, i.e., a $G$--torsor over $R[t_1,\ldots,t_n]$ is extended from $R$ if and only if for every maximal ideal ${\mathfrak m} \subset R$, the $G$--torsor on $R_{\mathfrak m}[t_1,\ldots,t_n]$ obtained by restriction is extended from $R_{\mathfrak m}$.
\end{cor}

\begin{proof}
Combine Proposition~\ref{prop:criteriaforlinearity} and Theorem~\ref{thm:localtoglobalprinciple}.
\end{proof}

\subsection{Affine homotopy invariance for $G$--torsors}
\label{ss:homotopyinvariance}
Let $G$ be a smooth linear $R$--group scheme.  In this section, we analyze when the pullback map
\[
H^1_{\Nis}(X,G)\longrightarrow H^1_{\Nis}(X\times\aone,G)
\]
is a bijection for $X$ a smooth affine $R$--scheme.

\subsubsection*{Special linear groups}
We begin by recalling some facts about oriented vector bundles over schemes.  If $X$ is a scheme, then recall that an oriented vector bundle on $X$ is a pair $(\mathscr{E},\varphi)$ consisting of a vector bundle $\mathscr{E}$ on $X$ equipped with an isomorphism $\varphi: \det \mathscr{E} \isomt \mathscr{O}_X$. There is a standard equivalence between the groupoid of oriented vector bundles on $X$ and that of $SL_n$--torsors over $X$.  Write $\mathscr{V}^o_{n}(X)$ for the set of isomorphism classes of rank $n$ oriented vector bundles on $X$.

\begin{thm}[Special linear homotopy invariance]
\label{thm:SLhomotopyinvariance}
Fix an integer $n \geq 1$ and suppose $R$ is a ring such that, for every maximal ideal ${\mathfrak m} \subset R$, $R_{{\mathfrak m}}$ is ind-smooth over a Dedekind ring with perfect residue fields (for example, $R_{{\mathfrak m}}$ is Noetherian and regular over such a Dedekind ring).  For every integer $m \geq 0$, the map
\[
\mathscr{V}^o_{n}(\Spec R) \longrightarrow \mathscr{V}^o_{n}(\Spec R[t_1,\ldots,t_m])
\]
is a bijection.
\end{thm}

\begin{proof}
By \cite[Theorem 5.2.1]{PartI}, every vector bundle on $\Spec R[t_1,\ldots,t_m]$ is pulled back from a vector bundle on $\Spec R$.  In particular, every oriented vector bundle on $\Spec R[t_1,\ldots,t_m]$ is pulled back from a vector bundle on $\Spec R$ with trivial determinant. It remains to show that every automorphism of the trivial line bundle on $\Spec R[t_1,\ldots,t_m]$ is extended from $\Spec R$.  In other words, we must show that the inclusion map $R \to R[t_1,\ldots,t_m]$ induces an isomorphism on unit groups.

Observe that our assumptions guarantee that $R_{{\mathfrak m}}$ is reduced for every maximal ideal ${\mathfrak m} \subset R$, and therefore $R$ must itself be reduced.  Since $R$ is reduced, the fact that $R \to R[t_1,\ldots,t_m]$ induces an isomorphism on unit groups follows from a straightforward induction argument, using the elementary observation that if $A$ is a reduced commutative ring, then the map $A \to A[t]$ induces an isomorphism $A^{\times} \to A[t]^{\times}$.
\end{proof}

\begin{rem}\label{rem:metalinear}
In \cite[Definition 4.3]{MField}, Morel defines an orientation on a vector bundle $\mathscr{E}$ to be an isomorphism between $\det \mathscr{E}$ and the square of a line bundle.  Oriented vector bundles in this sense correspond to torsors under the \emph{metalinear group} $ML_n$ defined by the pullback square
\[
\xymatrix{
ML_n \ar[r]\ar[d] & \mathbb G_m \ar[d]^{2} \\
GL_n \ar[r]_{\det} & \mathbb G_m.
}
\]
This more general notion of orientation is very natural in Morel's theory of the Euler class, since the latter only depends on an orientation in this sense.  
Theorem~\ref{thm:SLhomotopyinvariance} is also true for $ML_n$--torsors instead of $SL_n$--torsors, with a very similar proof.
\end{rem}

\subsubsection*{Symplectic groups}
We refer the reader to Knus \cite[\S I.4]{Knus} for more details about symplectic spaces over rings; we briefly fix notations in the scheme-theoretic context.  If $X$ is a scheme and $\mathcal{B}$ is a quasi-coherent sheaf on $X$, an {\em alternating bilinear form on $\mathcal{B}$} is a morphism of quasi-coherent sheaves $\varphi: \mathcal{B} \tensor_{\O_X} \mathcal{B} \to \O_X$ such that $\varphi\circ\Delta=0$, where  $\Delta: \mathcal{B} \to \mathcal{B} \tensor_{\O_X} \mathcal{B}$ is the (nonlinear) diagonal map.  If $(\mathcal{B},\varphi)$ is a quasi-coherent sheaf equipped with an alternating bilinear form, then we will say that $\varphi$ is {\em non-degenerate} if $\varphi$ induces an isomorphism $\mathcal{B} \to \mathcal{B}^{\vee} := \hom_{\O_X}(\mathcal{B},\O_X)$.  By a {\em symplectic bundle (of rank $2n$)} we will mean a pair $(\mathcal{B},\varphi)$ consisting of a (rank $2n$) vector bundle $\mathcal{B}$ on $X$ equipped with a non-degenerate alternating bilinear form $\varphi$.  Write $\mathscr{HV}_{2n}(X)$ for the set of isomorphism classes of rank $2n$ symplectic bundles on $X$.

We briefly recall the standard equivalence between the groupoid of symplectic vector bundles and that of $Sp_{2n}$--torsors on $X$.  In one direction, send a symplectic vector bundle $(\mathcal{B},\varphi)$ to its bundle of ``symplectic frames''; by \cite[Proposition I.4.1.4]{Knus} this construction yields an fppf torsor under $Sp_{2n}$.  In the other direction, given an $Sp_{2n}$--torsor $\mathscr{P}$ on $X$, consider the vector bundle associated with the standard $2n$--dimensional representation of $Sp_{2n}$, which comes equipped with a reduction of structure group to $Sp_{2n}$, i.e., an alternating form on the bundle.  By \cite[Corollary 4.1.2]{Knus} any symplectic bundle on a scheme $X$ is Zariski locally on $X$ isometric to the hyperbolic space of a trivial vector bundle \cite[I.3.5]{Knus}.  Combining these observations, we see that $Sp_{2n}$--torsors are Zariski locally trivial and that there is an equivalence between the groupoid of symplectic vector bundles over $X$ and that of Nisnevich locally trivial $Sp_{2n}$--torsors (as mentioned in Example~\ref{ex:bundlesandtorsors}).

\begin{thm}[Symplectic homotopy invariance]
\label{thm:sympletichomotopyinvariance}
Fix an integer $n \geq 1$ and suppose $R$ is a ring such that, for every maximal ideal ${\mathfrak m} \subset R$, $R_{{\mathfrak m}}$ is ind-smooth over a Dedekind ring with perfect residue fields (for example, $R_{{\mathfrak m}}$ is Noetherian and regular over such a Dedekind ring).  For every integer $m \geq 0$, the map
\[
\mathscr{HV}_{2n}(\Spec R) \longrightarrow \mathscr{HV}_{2n}(\Spec R[t_1,\ldots,t_m])
\]
is a bijection.
\end{thm}

\begin{proof}
For any integer $n\geq 1$, the group $Sp_{2n}$ is a split reductive $R$--group scheme (and, by definition, linear). Therefore, applying Theorem~\ref{thm:localtoglobalprinciple}, it suffices to demonstrate the result with $R$ replaced by $R_{{\mathfrak m}}$.  Since $R_{{\mathfrak m}}$ is local, every finitely generated projective module over $R_{{\mathfrak m}}$ is free.  By the assumption on $R$ and \cite[Theorem 5.2.1]{PartI}, we know that, for any integer $m$, every finitely generated projective $R_{{\mathfrak m}}[t_1,\ldots,t_m]$--module is free.  Applying \cite[Corollary I.4.1.2]{Knus}, we conclude that every symplectic space over $R_{{\mathfrak m}}[t_1,\ldots,t_m]$ is isometric to the hyperbolic space of a free module.  In particular, every symplectic space over $R_{{\mathfrak m}}[t_1,\ldots,t_m]$ is extended from $R_{{\mathfrak m}}$.
\end{proof}

\subsubsection*{A formalism for homotopy invariance}
We recall a formalism introduced by Colliot-Th\'el\`ene--Ojanguren; the following result is a slight extension of \cite[Th\'eor\`eme 1.1]{CTO92}.

\begin{prop}
\label{prop:CTOProposition1.5}
Fix an infinite base field $k$.  Suppose $\mathbf{F}$ is a functor from the category of $k$--algebras to the category of pointed sets with the following properties:
\begin{itemize}
\item[$\mathbf{P1}$] The functor $\mathbf{F}$ commutes with filtered inductive limits of rings with flat transition morphisms.
\item[$\mathbf{P2}$] For every extension field $L/k$ and every integer $n \geq 0$, the restriction map
\[
\mathbf{F}(L[t_1,\ldots,t_n]) \longrightarrow \mathbf{F}(L(t_1,\ldots,t_n))
\]
has trivial kernel.
\item[$\mathbf{P3}$] The functor $\mathbf{F}$ has {\em weak affine Nisnevich excision}, i.e., for any smooth $k$--algebra $A$, any \'etale $A$--algebra $B$, and any element $f \in A$ such that $A/fA \cong B/fB$ the map
    \[
    \ker(\mathbf{F}(A) \to \mathbf{F}(A_f)) \longrightarrow \ker(\mathbf{F}(B) \longrightarrow \mathbf{F}(B_f))
    \]
    is a surjection.
\end{itemize}
If $B$ is the localization of a smooth $k$--algebra at a maximal ideal, then, setting $K_B = \operatorname{Frac}(B)$, for any integer $n \geq 0$ the restriction map
\[
\mathbf{F}(B[t_1,\ldots,t_n]) \longrightarrow \mathbf{F}(K_B(t_1,\ldots,t_n))
\]
has trivial kernel.
\end{prop}

\begin{proof}
Set $d := \dim B$ and write ${\mathfrak m}$ for the maximal ideal of $B$.  Suppose that
$$
\xi_0 \in \ker(\mathbf{F}(B[t_1,\ldots,t_n]) \longrightarrow \mathbf{F}(K_B(t_1,\ldots,t_n))).
$$
Let $\xi$ be the image of $\xi_0$ in $\mathbf{F}(K_B[t_1,\ldots,t_n])$.  Then, by assumption, $\xi$ lies in the kernel of $\mathbf{F}(K_B[t_1,\ldots,t_n]) \to \mathbf{F}(K_B(t_1,\ldots,t_n))$. By $\mathbf{P2}$, we conclude that $\xi$ is trivial.

By using $\mathbf{P1}$, we conclude that there is an element $g \in {\mathfrak m} \setminus 0$ such that $\xi_0$ restricts trivially to $\mathbf{F}(B_g[t_1,\ldots,t_n])$.  Then, by Knus \cite[Corollary VIII.3.2.5]{Knus}, there exist a polynomial ring $L[x_1,\ldots,x_d]$, a maximal ideal ${\mathfrak n} \subset L[x_1,\ldots,x_d]$, a local essentially \'etale morphism $\varphi: A \to B$ (where $A = L[x_1,\ldots,x_d]_{{\mathfrak n}}$), and an element $f \in \mathfrak{m}$ such that $\varphi(f) = ug$ for $u$ a unit in $B_{{\mathfrak m}}$ and $\varphi$ induces an isomorphism $A/fA \isomt B/gB$.  By $\mathbf{P3}$, we conclude that there exists an element $\xi'_0 \in \ker(\mathbf{F}(A[t_1,\ldots,t_n]) \to \mathbf{F}(A_{f}[t_1,\ldots,t_n]))$ mapping to $\xi_0$.  However, $\xi'_0$ is also evidently in $\ker(\mathbf{F}(A[t_1,\ldots,t_n]) \to \mathbf{F}(K_A(t_1,\ldots,t_n)))$.  Thus, it suffices to establish the result in the case where $B$ is the localization of a polynomial ring at a maximal ideal, which is precisely \cite[Proposition 1.5]{CTO92}.
\end{proof}

\subsubsection*{Isotropic reductive groups}
If $k$ is a field, a reductive $k$--group scheme will be called {\em anisotropic} if it contains no $k$--subgroup isomorphic to $\gm{}$.  We take the following definition for isotropic reductive $k$--group, but we caution the reader that our definition differs from that of Borel \cite[Definition V.20.1]{Borel}; we choose this definition because it better suits our eventual applications.

\begin{defn}
\label{defn:isotropic}
If $k$ is a field, a reductive $k$--group scheme $G$ will be called {\em isotropic} if each of the almost $k$--simple components of the derived group of $G$ contains a $k$--subgroup scheme isomorphic to $\gm{}$.
\end{defn}

\begin{rem}
See Borel \cite[\S V.20]{Borel} or Gille \cite[\S 9.1]{GilleLuminy} for further discussion of isotropic reductive groups.  In general, the existence of a {\em non-central} split multiplicative $k$--subgroup is equivalent to the existence of a parabolic $k$--subgroup by the dynamic construction described just before Lemma~\ref{lem:parabolics}.  In particular, isotropic reductive $k$--groups admit proper parabolic subgroups.
\end{rem}

\begin{thm}
\label{thm:hoinvisotropic}
If $k$ is an infinite field, and $G$ is an isotropic reductive $k$--group (see \textup{Definition~\ref{defn:isotropic}}), then for any smooth $k$--algebra $A$ and any integer $n \geq 0$, the map
\[
H^1_{\Nis}(\Spec A,G) \longrightarrow H^1_{\Nis}(\Spec A[t_1,\ldots,t_n],G)
\]
is a bijection.
\end{thm}

\begin{proof}
We have to show that every Nisnevich locally trivial $G$--torsor $\mathscr P$ over $A[t_1,\ldots,t_n]$ is extended from $A$.
After Corollary~\ref{cor:localtoglobalisotropic}, it suffices to show that, for every maximal ideal ${\mathfrak m}$ of $A$, the $G$--torsor $\mathscr P_{\mathfrak m}$ over $A_{\mathfrak m}[t_1,\ldots,t_n]$ is extended from $A_{\mathfrak m}$; we will show that in fact $\mathscr P_{\mathfrak m}$ is trivial.

We claim that the functor $A \mapsto H^1_{\Nis}(\Spec A,G)$ from $k$--algebras to pointed sets satisfies the axioms $\mathbf {P1}-\mathbf {P3}$ of Proposition~\ref{prop:CTOProposition1.5}.  Axiom $\mathbf{P1}$ is a consequence of our finite presentation hypotheses by way of Lemma~\ref{lem:finitepresentation}.  Axiom $\mathbf{P2}$ uses the hypothesis that $G$ is isotropic and follows from \cite[Proposition 2.4 and Theorem 2.5]{CTO92} (note that our definition of isotropic reductive $k$--group coincides with that used in \cite[\S 2 p.~103]{CTO92}).  Axiom $\mathbf{P3}$ is a formal consequence of the fact that $H^1_{\Nis}(-,G)\cong \pi_0(B\mathbf{Tors}_{\Nis}(G))$ where $B\mathbf{Tors}_{\Nis}(G)$ satisfies affine Nisnevich excision (see Section~\ref{ss:torsorsrepresentable}).  By the conclusion of Proposition~\ref{prop:CTOProposition1.5}, it suffices to show $\mathscr P_{\mathfrak m}$ becomes trivial over $\operatorname{Frac}(A_{\mathfrak m})(t_1,\ldots,t_n)$, but this follows immediately from the fact that a field has no nontrivial Nisnevich covering sieves.
\end{proof}

\begin{rem}
\label{rem:hoinvfailure}
At least if $k$ is an infinite perfect field, Theorem~\ref{thm:hoinvisotropic} admits a converse: if $G$ is a reductive $k$--group such that $H^1_{\Nis}(-,G)$ is $\aone$--invariant on $\Sm_k^\aff$, then $G$ is isotropic, see Balwe and Sawant \cite[Theorem 1]{BalweSawantReductive}. In fact, for $G$ reductive, the following three conditions are equivalent:
\begin{enumerate}
\item[(i)] $G$ is isotropic (in the sense of Definition~\ref{defn:isotropic});
\item[(ii)] $H^1_{\Nis}(-,G)$ is $\aone$--invariant on smooth affine $k$--schemes;
\item[(iii)] $R_{\Nis}\Singaone G$ is $\aone$--invariant.
\end{enumerate}
The implication (i) $\Rightarrow$ (ii) is Theorem~\ref{thm:hoinvisotropic}, (ii) $\Rightarrow$ (iii) is a special case of Theorem~\ref{thm:homogeneousrep}, and (iii) $\Rightarrow$ (i) is \cite[Theorem 4.7]{BalweSawantReductive}.
\end{rem}

\section{Applications}
\label{sec:applications}
In this section, we collect a number of applications of the results established so far.  Section~\ref{ss:reptorsors} discusses representability results for Nisnevich locally trivial torsors.  As mentioned in Remark \ref{remintro:specialgroups}, that representability results should hold for torsors under $SL_n$ and $Sp_{2n}$ was observed by Schlichting \cite[Remark 6.23]{Schlichting}; we simply observe in these cases that the expected classical geometric objects yield models for the representing spaces.  In Section~\ref{ss:rephomogeneousspaces} we establish that for various classes of homogeneous spaces for reductive groups applying the singular construction produces an $\aone$--local space.  Section~\ref{ss:strongaoneinvariance} establishes strong $\aone$--invariance of homotopy sheaves of the singular construction of a reductive group under suitable additional hypotheses.  Finally, Section~\ref{ss:nilpotence} studies a purely algebraic problem using our techniques, namely nilpotence of non-stable $K_1$--functors.

\subsection{Affine representability results for torsors}
\label{ss:reptorsors}
Let $\mathrm{Gr}_{n,n+N}$ be the usual Grassmannian parameterizing  $n$--dimensional subspaces of an $(n+N)$--dimensional vector space.  Let $\widetilde{\mathrm{Gr}}_{n,n+N}$ be the complement of the zero section in the total space of the determinant of the tautological vector bundle on $\mathrm{Gr}_{n,n+N}$.  The space $\widetilde{\mathrm{Gr}}_{n,n+N}$ parameterizes rank $n$ subspaces of the $(n+N)$--dimensional vector space equipped with a specified trivialization of their determinant.  We set $\widetilde{\mathrm{Gr}}_n := \colim_N \widetilde{\mathrm{Gr}}_{n,n+N}$ where the transition maps are the same as those in the definition $\mathrm{Gr}_n$.  With these definitions, we can establish a geometric representability result for oriented vector bundles.

\begin{thm}
\label{thm:orientedrepresentability}
Suppose $k$ is ind-smooth over a Dedekind ring with perfect residue fields. Then, for any $X\in\Sm_k^\aff$, and any integer $n \geq 1$, there is a bijection
\[
\mathscr{V}_{n}^o(X) \cong [X,\widetilde{\mathrm{Gr}}_n]_{\aone}
\]
that is functorial in $X$.
\end{thm}

\begin{proof}
Recall from Example~\ref{ex:bundlesandtorsors} and the discussion preceding Theorem~\ref{thm:SLhomotopyinvariance} that, for any integer $n \geq 1$, there is a functorial bijection of the form $\mathscr {V}_{n}^o(X)\cong H^1_{\Nis}(X, SL_{n})$.  Combining Theorems~\ref{thm:Gtorsors} and \ref{thm:SLhomotopyinvariance}, we conclude that, under the stated hypotheses on $k$, for any smooth affine $k$--scheme $X$, $H^1_{\Nis}(X, SL_{n}) \cong [X,BSL_{n}]_{\aone}$.

Using the notation of Morel and Voevodsky \cite[\S 4.2]{MV}, the space $B_{gm}(SL_n,i)$ (attached to the defining inclusion $i: SL_n \hookrightarrow GL_n$) is precisely the space $\widetilde{\mathrm{Gr}}_n$.  Therefore combining the results of \cite[\S 4.2]{MV}, and using the fact that all $SL_n$--torsors are Zariski (and thus Nisnevich) locally trivial we conclude that the map $\widetilde{\mathrm{Gr}}_n \to BSL_n$ classifying the universal $SL_n$--torsor over $\widetilde{\mathrm{Gr}}_n$ is an $\aone$--weak equivalence.
\end{proof}

If we let $\mathrm{H}$ be the standard $2$--dimensional hyperbolic space, then we can consider the symplectic vector space $\mathrm{H}^{\oplus N}$.  Panin and Walter construct a scheme $\mathrm{HGr}_{n,n+N}$ that parameterizes rank $2n$ symplectic subspaces of $\mathrm{H}^{\oplus (n+N)}$ and we set $\mathrm{HGr}_n := \colim_N \mathrm{HGr}_{n,n+N}$ \cite{PaninWalterPontryaginClasses}.  Alternatively, $\mathrm{HGr}$ can be described as the colimit $\colim_N Sp_{2(n+N)}/(Sp_{2n} \times Sp_{2N})$.  Using these definitions, we are now able to establish a geometric representability theorem for symplectic vector bundles.

\begin{thm}
\label{thm:symplecticrepresentability}
Suppose $k$ is ind-smooth over a Dedekind ring with perfect residue fields. Then, for any $X\in\Sm_k^\aff$, there is a bijection
\[
\mathscr{HV}_{2n}(X) \cong [X,\mathrm{HGr}_n]_{\aone}
\]
that is functorial in $X$.
\end{thm}

\begin{proof}
Proceeding as in the proof of Theorem~\ref{thm:orientedrepresentability}, we combine Example~\ref{ex:bundlesandtorsors} and the discussion preceding Theorem~\ref{thm:sympletichomotopyinvariance} to conclude that there is a functorial bijection of the form $\mathscr {HV}_{2n}(X)\cong H^1_{\Nis}(X, Sp_{2n})$.  Combining Theorems~\ref{thm:Gtorsors} and \ref{thm:sympletichomotopyinvariance}, we conclude that, under the stated hypotheses on $k$, for any smooth affine $k$--scheme $X$, $H^1_{\Nis}(X, Sp_{2n}) \cong [X,BSp_{2n}]_{\aone}$.  Finally, by the proof of \cite[Theorem 8.2]{PaninWalterBO}, we can conclude that $\mathrm{HGr}_n$ is $\aone$--weakly equivalent to  $BSp_{2n}$, and thus for any smooth $k$--scheme $X$, $[X,\mathrm{HGr}_n]_{\aone} \cong [X,BSp_{2n}]_{\aone}$.
\end{proof}

We now establish Theorem~\ref{thmintro:isotropictorsors}.

\begin{thm}
\label{thm:isotropictorsors}
Suppose $k$ is an infinite field, and $G$ is an isotropic reductive $k$--group (see \textup{Definition~\ref{defn:isotropic}}).  For any smooth affine $k$--scheme $X$, there is a functorial bijection
\[
H^1_{\Nis}(X,G) \cong [X,BG]_{\aone}.
\]
\end{thm}

\begin{proof}
Combine Theorems~\ref{thm:Gtorsors} and \ref{thm:hoinvisotropic}.
\end{proof}

\begin{rem}
In Theorem~\ref{thm:isotropictorsors}, the isotropy condition on $G$ cannot be weakened, cf. Remark~\ref{rem:hoinvfailure}.
\end{rem}

\subsection{Affine representability results for some homogeneous spaces}
\label{ss:rephomogeneousspaces}
Let $Q_{2n-1}$ be the smooth affine quadric over $\Z$ defined by $\sum_i x_i y_i = 1$.  There is a standard identification $SL_n/SL_{n-1} \isomt Q_{2n-1}$.  Let $Q_{2n}$ be the smooth affine quadric over $\Z$ defined by $\sum_i x_iy_i = z(z+1)$ (in Asok--Doran--Fasel \cite{AsokDoranFasel}, it is shown that $Q_{2n}$ is $\aone$--weakly equivalent to ${\pone}^{\sma n}$ over $\Spec \Z$).  In particular, there are isomorphisms $Q_2 \cong SL_2/\gm{}$ and $Q_4 \cong Sp_4/(Sp_2 \times Sp_2)$ over $\Spec \Z$.  If $R$ is a ring in which $2$ is invertible, then $Q_{2n}$ is isomorphic over $R$ to the quadric defined by the standard hyperbolic form $\sum_i x_iy_i + z^2 = 1$.  It then follows from Lemma~\ref{lem:orthogonalquotient} that $Q_{2n}$ is isomorphic over $R$ to the homogeneous space $SO_{2n+1}/SO_{2n}$.

\begin{thm}
\label{thm:oddquadric}
If $R$ is a ring that is ind-smooth over a Dedekind ring with perfect residue fields, then $Q_{2n-1}$ is $\A^1$--naive.  In particular, for any smooth affine $R$--scheme $X$, there is a functorial bijection
\[
\pi_0(\Singaone Q_{2n-1})(X) \isomto [X,Q_{2n-1}]_{\aone}.
\]
\end{thm}

\begin{proof}
The scheme $Q_{2n-1}$ is isomorphic over $\Spec \Z$ to the homogeneous space $GL_n/GL_{n-1}$.  Since all torsors for $GL_{n-1}$ are Zariski locally trivial, it follows that $GL_n \to Q_{2n-1}$ is Zariski locally trivial (in fact, one can just write down an explicit trivialization).   Using \cite[Theorem 5.2.1]{PartI} we may apply Theorem~\ref{thm:homogeneousrep} to conclude.
\end{proof}

\begin{thm}
\label{thm:evenquadric}
If either (a) $n \leq 2$, and $R$ is a ring that is ind-smooth over a Dedekind ring with perfect residue fields, or (b) $n \geq 3$ and  $R$ is an infinite field having characteristic unequal to $2$, then $Q_{2n}$ is $\A^1$--naive.  In particular, under either set of hypotheses,
for any smooth affine $R$--scheme $X$, there is a functorial bijection
\[
\pi_0(\Singaone Q_{2n})(X) \isomto [X,Q_{2n}]_{\aone}.
\]
\end{thm}

\begin{proof}
For $n = 1$ consider the identification $Q_2 \cong SL_2/\gm{}$. Affine homotopy invariance holds for $\gm{}$--torsors over an arbitrary regular base, and for torsors under $SL_2\cong Sp_2$  by assumption.  The result follows immediately from Theorem~\ref{thm:homogeneousrep}.  Similarly, for $n = 2$ consider the identification $Q_4 \cong Sp_4/(Sp_2 \times Sp_2)$.  Again, by assumption we may combine Theorems~\ref{thm:sympletichomotopyinvariance} and \ref{thm:homogeneousrep} to conclude.

For $n \geq 3$ we proceed slightly differently.  The $SO_{2n}$--torsor $SO_{2n+1} \to Q_{2n}$ is still Zariski locally trivial by Lemma~\ref{lem:orthogonalquotient}.  Since $SO_m$ is split for $m \geq 3$, we may apply Theorem~\ref{thm:hoinvisotropic} to conclude that $H^1_{\Nis}(-,SO_{m})$ is $\aone$--invariant on $\Sm_R^\aff$ for any integer $m \geq 3$.  Then, we apply Theorem~\ref{thm:homogeneousrep} to conclude.
\end{proof}

\begin{rem}
If $X = \Spec A$, then a map $f:X \to Q_{2n}$ yields an ideal $I \subset A$ and a surjection $\omega: (A/I)^{\oplus n} \to I/I^2$; the ideal $I$ is the ideal generated by the images of $x_1,\ldots,x_n,z$ in the coordinate presentation of the quadric.  The class of $f$ in $\pi_0(\Singaone Q_{2n})(X)$ depends only on the pair $(I,\omega)$ and is called the ``Segre class'' of $(I,\omega)$, see Fasel \cite[Theorem 2.0.2]{FaselComplete}. When $X$ is smooth over an infinite field, the Segre class provides an obstruction to lifting $\omega$ to a surjection $A^{\oplus n} \to I$ \cite[Theorem 3.2.8]{FaselComplete}.
\end{rem}

\subsubsection*{Zariski fiber bundles with affine space fibers}
If $F$ is a fixed $S$--scheme, we will say that an $S$--morphism $\pi: E \to B$ is a {\em Zariski fiber bundle of $S$--schemes with fibers isomorphic to $F$} if there exist an $S$--scheme $U$, a Zariski covering morphism $U \to B$ and an isomorphism $\varphi: U \times_B E \isomt U \times_S F$ over $U$.  The following result, which generalizes a result of Morel \cite[Theorem 8.9(2)]{MField}, applies to affine vector bundle torsors (a.k.a. Jouanolou--Thomason devices, see Weibel \cite[Definition 4.2 and Proposition 4.4]{WeibelHAK}).

\begin{lem}
\label{lem:affineinvarianceofsing}
Suppose $B \in \Sm_S$, and $\pi: E \to B$ is a Zariski fiber bundle of $S$--schemes with fibers isomorphic to ${\mathbb A}^n_S$.  For any $X=\Spec R\in \Sm_S^{\aff}$, the induced map
\[
\Singaone E(X) \longrightarrow \Singaone B(X)
\]
is an acyclic Kan fibration.  Moreover, $E$ is $\A^1$--naive if and only if $B$ is $\A^1$--naive.
\end{lem}

\begin{proof}
By Goerss and Jardine \cite[Theorem I.11.2]{GoerssJardine}, it suffices to show that for any integer $n \geq 0$, given a diagram of the form
\[
\xymatrix{
\partial \Delta^n_R \ar[r]\ar[d] & E \ar[d]^{\pi} \\
\Delta^n_R \ar[r] & B
}
\]
there is a morphism $\Delta^n_R \to E$ making both resulting triangles commute.

Given a diagram as above, there is an induced map $\partial \Delta^n_R \to \Delta^n_R \times_B E$.  By the assumption on $\pi$, the pullback $\pi': \Delta^n_R \times_B E \to \Delta^n_R$ makes the ring of functions on $\Delta^n_R \times_B E$ into a locally polynomial algebra over $R[t_1,\ldots,t_n]$ in the sense of Bass--Connell--Wright \cite[Theorem 4.4]{BCW}.  Therefore, by \cite[Theorem 4.4]{BCW} we conclude that $\pi'$ is a geometric vector bundle over $\Delta^n_R$, i.e., the spectrum of a symmetric algebra over $\Delta^n_R$.
Now, if $\mathscr{E} \to \Delta^n_R$ is a geometric vector bundle, then the inclusion map $\partial \Delta^n_R \to \Delta^n_R$ induces a surjective map $\hom(\Delta^n_R,\mathscr{E}) \to \hom(\partial \Delta^n_R,\mathscr{E})$.  Therefore, the lift we hoped to construct is guaranteed to exist.

For the second statement, let $\tilde E$ and $\tilde B$ be Nisnevich-local $\A^1$--invariant replacements of $E$ and $B$, respectively, and consider the commutative square of simplicial presheaves
\[
\xymatrix{
\Singaone E \ar[r]\ar[d] & \tilde E \ar[d] \\
\Singaone B \ar[r] & \tilde B.
}
\]
Since the left vertical map is a weak equivalence on affines, the right vertical map is a weak equivalence. It follows that the upper horizontal map is a weak equivalence on affines if and only if the lower horizontal map is.
\end{proof}

\begin{ex}
If $X \in \Sm_S^{\aff}$ is an affine scheme, then any finitely presented Zariski fiber bundle of $S$--schemes $\pi: E \to X$ with fibers isomorphic to affine spaces is actually a vector bundle by the result of Bass--Connell--Wright mentioned above \cite{BCW}; this result was obtained independently by Suslin \cite{Suslinlocallypolynomial}.  On the other hand, if $X$ is not affine, then even if $\pi$ admits a section, it may not be isomorphic to a vector bundle: see Iarrobino \cite[Theorem 1]{Iarrobino} for an example with $X = \pone$.
\end{ex}

\subsubsection*{Homogeneous spaces with non-reductive stabilizers}
The following result extends and simplifies the proof of a theorem of Morel \cite[Theorem 8.9]{MField} (in particular, we allow the case $n = 2$).

\begin{cor}
\label{cor:SingAnMinusZero}
If $R$ is a ring that is ind-smooth over a Dedekind ring with perfect residue fields, then ${\mathbb A}^n \setminus 0$ is $\A^1$--naive.  In particular, for any smooth affine $R$--scheme $X$, there is a canonical bijection
\[
\pi_0\Sing^{\A^1} ({\mathbb A}^n \setminus 0)(X) \isomto [X,{\mathbb A}^n \setminus 0]_{\aone}.
\]
\end{cor}

\begin{proof}
The map $SL_n \to {\mathbb A}^n \setminus 0$ given by ``projection onto the first column'' factors through a map $SL_n/SL_{n-1} \to {\mathbb A}^n \setminus 0$; this map is a Zariski fiber bundle with fibers isomorphic to affine spaces.  By Lemma~\ref{lem:affineinvarianceofsing}, it suffices to show that $SL_n/SL_{n-1}$ is $\A^1$--naive.  This follows from Theorem~\ref{thm:oddquadric} via the standard isomorphism $SL_n/SL_{n-1} \cong Q_{2n-1}$ (send a matrix in $SL_n$ to the first row and first column of its inverse).
\end{proof}

\begin{lem}\label{lem:connectedness}
	Let $X$ be a simplicial set and $k\geq 0$. If $X$ has the right lifting property with respect to the inclusion $\partial\Delta^m\subset\Delta^m$ for every $m\leq k+1$, then $X$ is $k$--connected.
\end{lem}

\begin{proof}
	A simplicial set $X$ is $k$--connected if and only if the Kan complex $\mathrm{cosk}_{k+1}\mathrm{Ex}^\infty X$ is contractible, or equivalently has the right lifting property with respect to $\partial\Delta^m\subset \Delta^m$ for all $m$.
	By adjunction, this is the case if and only if $\mathrm{Ex}^\infty X$ has the right lifting property with respect to $\partial\Delta^m\subset \Delta^m$ for $m\leq k+1$. By definition of $\mathrm{Ex}^\infty$, it suffices to show that $X$ itself has the right lifting property with respect to $\mathrm{sd}^r(\partial\Delta^m)\subset\mathrm{sd}^r(\Delta^m)$ for all $r$ and all $m\leq k+1$. In fact, $X$ has the right lifting property with respect to any monomorphism between $(k+1)$--skeletal simplicial sets, since such a monomorphism is a transfinite composition of pushouts of $\partial\Delta^m\subset\Delta^m$ for $m\leq k+1$.
\end{proof}

\begin{prop}
\label{prop:connectedness}
Let $n,k\geq 0$ and let $R$ be a commutative ring such that the Bass stable range of $R[t_0,\dotsc,t_k]$ is at most $n$.
Then the simplicial set $\Sing^{\A^1}(\A^n\setminus 0)(R)$ is $k$--connected.
In particular, if $R$ is Noetherian of Krull dimension $d$, then $\Sing^{\A^1}(\A^n\setminus 0)(R)$ is $(n-d-2)$--connected.
\end{prop}

\begin{proof}
By Lemma~\ref{lem:connectedness}, it suffices to show that the map
\[
Um_n(\Delta^m_R) \to Um_n(\partial\Delta^m_R)
\]
is surjective for all $m\leq k+1$, where $Um_n(X)=\hom(X,\A^n\setminus 0)$ is the set of unimodular rows of length $n$ in $\O(X)$. By assumption, the Bass stable range of $\Delta^{k+1}_R$ is at most $n$. It follows that the Bass stable range of $\Delta^m_R$ is at most $n$, for all $m\leq k+1$.
Now the result is a special case of the following more general statement, which follows easily from the definition of Bass stable range: if $X$ is an affine scheme of Bass stable range $\leq n$ and $Y\subset X$ is a finitely presented closed subscheme, then the map $Um_n(X) \to Um_n(Y)$ is surjective.
\end{proof}

\begin{rem}
Under the assumption of Corollary~\ref{cor:SingAnMinusZero}, if $n \geq 3$, the set $\pi_0\Sing^{\A^1} ({\mathbb A}^n \setminus 0)(X)$ has a concrete description due to Fasel \cite[Theorem 2.1]{FaselUnimodular}.  Indeed, it is the quotient of the set $Um_{n}(X)$ of unimodular rows of length $n$ by the action of the subgroup $E_n(X) \subset SL_n(X)$ generated by elementary shearing matrices. In \textit{loc.~cit.}, it is assumed that $R$ is a field, but the proof works more generally using a result of Lindel--Popescu \cite[Proposition 2.1]{Popescu}. Taking $X=Q_{2n-1}$, we obtain a bijection
\[
[\A^n\setminus 0,\A^n\setminus 0]_{\A^1}\cong Um_n(Q_{2n-1})/E_n(Q_{2n-1}).
\]
By Corollary~\ref{cor:SingAnMinusZero}, we have $[S^1,{\mathbb A}^n \setminus 0]_{\aone,*} \cong \pi_1\Sing^{\A^1} ({\mathbb A}^n \setminus 0)(R)$, and Proposition~\ref{prop:connectedness} shows that this group is trivial if $n$ is at least the Bass stable range of $R[t_0,t_1]$. In that case, we may therefore identify $[\A^n\setminus 0,\A^n\setminus 0]_{\A^1}$ with the set of maps in the \emph{pointed} $\A^1$--homotopy category.  Note that $\colim_n [\A^n\setminus 0,\A^n\setminus 0]_{\A^1,*}$ is the set of endomorphisms of the motivic sphere spectrum over the ring $R$.
\end{rem}

The following result is Theorem~\ref{thmintro:parabolics}.

\begin{thm}
\label{thm:parabolics}
If $k$ is an infinite field, $G$ is an isotropic reductive $k$--group (see \textup{Definition~\ref{defn:isotropic}}) and $P \subset G$ is a parabolic $k$--subgroup possessing an isotropic Levi factor (e.g., if $G$ is split), then $G/P$ is $\A^1$--naive. In particular, for any smooth affine $k$--scheme $X$, there is a functorial bijection
\[
\pi_0(\Singaone G/P)(X) \isomto [X,G/P]_{\aone}.
\]
\end{thm}

\begin{rem}
Given a reductive $k$--group and a non-trivial parabolic subgroup $P \subset G$, it is not obvious that $P$ has a Levi factor.  Nevertheless, as mentioned in Remark~\ref{rem:levisexist}, our hypotheses guarantee that $P$ has a Levi factor.  If $L$ is a Levi factor for $P$, then $L$ may itself be anisotropic.
\end{rem}

\begin{proof}
Lemma~\ref{lem:parabolics}(ii) implies that $G \to G/L$ is generically trivial.  Since $k$ is assumed infinite and $L$ is reductive, we claim $G \to G/L$ is actually Zariski locally trivial.  An elementary argument for Zariski local triviality of $G \to G/L$ sketched in Remark~\ref{rem:levisexist}(2), but alternatively we can use \cite[Th\'eor\`eme 2.1]{CTO92}, to which, momentarily, implicit appeal will be made.

By Theorem~\ref{thm:homogeneousrep}, whose hypotheses hold by Theorem~\ref{thm:hoinvisotropic}, we conclude that $G/L$ is $\A^1$--naive.
By Lemma~\ref{lem:parabolics}(iii), $G/L \to G/P$ is a composition of Zariski fiber bundles with affine space fibers.  Hence, $G/P$ is also $\A^1$--naive by Lemma~\ref{lem:affineinvarianceofsing}.
\end{proof}

The above result can be significantly strengthened at the expense of further restrictions on the groups under consideration.

\begin{thm}
\label{thm:homogeneousspacesunderspecialgroups}
Suppose $R$ is ind-smooth over a Dedekind ring with perfect residue fields (for example, $R$ is Noetherian and regular over such a Dedekind ring).  If $G \cong GL_n$ or $Sp_{2n}$, and if $P \subset G$ is a standard parabolic subgroup, then $G/P$ is $\A^1$--naive. In particular, for any smooth affine $R$--scheme $X$, there is a functorial bijection
\[
\pi_0(\Singaone G/P)(X) \isomto [X,G/P]_{\aone}.
\]
\end{thm}

\begin{proof}
Assume first that $R = \Z$.  If $P \subset G$ is a standard parabolic with Levi factor $L$, then $L$ is itself a special group in the sense of Grothendieck--Serre, i.e., all \'etale locally trivial torsors are Zariski locally trivial.  Thus, the map $G \to G/L$ in Lemma~\ref{lem:parabolics}(ii) is automatically Zariski locally trivial.  One sees that the map $G/L \to G/P$ is a Zariski fiber bundle with affine space fibers by combining Lemma~\ref{lem:parabolics}(iii) with the fact that all finitely generated projective $\Z$--modules are free.  By extending scalars to $R$, it follows that corresponding statements hold for the resulting group scheme over $R$.

With these modifications, the proof is essentially identical to that of Theorem~\ref{thm:parabolics}; however, instead of appealing to Theorem~\ref{thm:hoinvisotropic}, we use Theorem~\ref{thm:sympletichomotopyinvariance} or \cite[Theorem 5.2.1]{PartI} to establish the necessary homotopy invariance statement.
\end{proof}

\begin{ex}
\label{ex:poneoveraring}
Theorem \ref{thm:homogeneousspacesunderspecialgroups} applies if $P \subset GL_n$ is a maximal parabolic subgroup, in which case $G/P \cong \mathrm{Gr}_{m,n}$ for some integer $m \leq n$.
\end{ex}

\subsection{Affine representability for non-stable K-theory and strong $\A^1$--invariance results}
\label{ss:strongaoneinvariance}
Suppose $G$ is a smooth linear $R$--group scheme.  For any integer $i \geq 1$, one can define Karoubi--Villamayor-style non-stable K-theory functors attached to $G$ by means of the formula:
\[
KV_{i+1}^G(U) := \pi_i(\Singaone G)(U)
\]
In this form, the definition goes back to Jardine \cite[Theorem 3.8]{Jardineunstablektheory}, but had precursors in the work of Krusemeyer \cite[\S 3]{Krusemeyer}; see Wendt \cite{WendtChevalley} for a more detailed analysis of such functors in the context of $\aone$--homotopy theory.  As a straightforward application of our results, we obtain $\aone$--representability results for non-stable $KV$--functors.

\begin{thm}
\label{thm:isotropicaonelocal}
If $k$ is an infinite field, and $G$ is an isotropic reductive $k$--group (in the sense of \textup{Definition~\ref{defn:isotropic}}), then $G$ is $\A^1$--naive.
In particular, for any smooth affine $k$--scheme $U$, there are canonical isomorphisms
\[
KV_{i+1}^G(U) \cong [S^i \wedge U_+,G]_{\aone,*}.
\]
\end{thm}

\begin{proof}
Apply Theorem~\ref{thm:homogeneousrep} with $H = e$ (hypotheses being satisfied by Theorem~\ref{thm:hoinvisotropic}).
\end{proof}

\begin{rem}
\label{rem:isotropicaonelocal}
Results such as the above were studied initially by Morel \cite[Theorem 8.1]{MField} and Moser \cite{MoserSL2} (see also \cite[Theorem 5.3]{WendtRationallyTrivial}) for $G$ a general split group, and by the third author and K V{\"o}lkel in the isotropic reductive case \cite{VoelkelWendt}. These results depend crucially on first establishing homotopy invariance for non-stable $K_1$--functors via  ``elementary matrix'' techniques.  As a consequence these proofs do not easily extend to the important case where $G$ has semi-simple rank $1$, which was treated separately by Moser.  Our proof above makes no such assumption on the homotopy invariance of non-stable $K_1$--functors.  As a consequence, Theorem~\ref{thm:isotropicaonelocal} can also be used to slightly uniformize the proof of \cite[Theorem 3.4]{BalweSawant}.
\end{rem}

We can also establish the strong $\aone$--invariance of the sheafifications of the non-stable $K_1$--presheaves attached to arbitrary isotropic reductive $k$--groups with $k$ infinite.

\begin{thm}
\label{thm:wha1invariance}
Suppose $k$ is an infinite field, and $G$ is an isotropic reductive $k$--group (in the sense of \textup{Definition~\ref{defn:isotropic}}).  For any integer $n \geq 0$, the following statements hold.
\begin{enumerate}
\item[(i)] The Zariski sheaf $a_{\Zar}\pi_n(\Singaone G)$ is a Nisnevich sheaf.
\item[(ii)] The sheaf $a_{\Zar}\pi_n(\Singaone G)$ is strongly $\aone$--invariant.
\end{enumerate}
\end{thm}

\begin{proof}
We begin by recalling some key results of Morel \cite[Chapter 6]{MField}.  If $\mathscr{X}$ is a Nisnevich-local and $\aone$--invariant pointed simplicial presheaf on $\Sm_k$, the sheaf $a_{\Nis}\pi_1(\mathscr{X})$ is strongly $\aone$--invariant by \cite[Theorem 6.1]{MField}. Moreover, the map $a_{\Zar}\pi_1(\mathscr{X}) \to a_{\Nis}\pi_1(\mathscr{X})$ is an isomorphism by \cite[Corollary 6.9(2)]{MField} (the standing assumption that $a_{\Nis}\pi_0(\mathscr X)$ is trivial is not used in the proof).

By Theorems~\ref{thm:Gtorsors}(i) and \ref{thm:hoinvisotropic}, under the stated hypotheses on $k$, the simplicial presheaf $R_{\Zar}\Singaone B_{\Nis} G$ is Nisnevich-local and $\A^1$--invariant.  Applying the results of the previous paragraph to the simplicial presheaf
\[
\mathscr X=\mathbf R\Omega^n R_{\Zar}\Singaone B_{\Nis}G,
\]
we conclude that
\[
a_{\Zar}\pi_{n+1}(\Singaone B_{\Nis}G)
\]
is a strongly $\A^1$--invariant Nisnevich sheaf of groups for any $n\geq 0$. By Corollary~\ref{cor:SingOmega}, the map
\[
\pi_{n}(\Singaone \mathbf R\Omega B_{\Nis}G) \longrightarrow \pi_{n+1}(\Singaone B_{\Nis}G)
\]
is an isomorphism on affines, and hence it becomes an isomorphism after Zariski sheafification.  Finally, we conclude the proof by observing that $G\simeq \mathbf R\Omega B_{\Nis}G$ by Lemma~\ref{lem:BG} (iii).
\end{proof}

\begin{rem}
\label{rem:strictaoneinvariance}
We note that the results from \cite[Chapter 6]{MField} used in the proof of Theorem~\ref{thm:wha1invariance} do not require $k$ to be perfect. If the base field $k$ is in addition perfect, then, provided $a_{\Zar}\pi_n(\Singaone G)$ is abelian, we can use \cite[Theorem 5.46]{MField} to conclude that it is strictly $\A^1$--invariant.  The assumption that $k$ is infinite in the above statement only appears because of our appeal to Theorem~\ref{thm:hoinvisotropic}.  To remove this restriction, we would need homotopy invariance for torsors under isotropic reductive groups over finite fields.
\end{rem}

If $G$ is a reductive $k$--group, we can define $G(k)^+$ to be the normal subgroup of $G(k)$ generated by the $k$--points of subgroups of $G$ isomorphic to $\mathbb{G}_a$.  The Whitehead group of $G$ is defined by the formula
\[
W(k,G):=G(k)/G(k)^+;
\]
we refer the reader to P Gille's survey \cite{GilleBourbaki} for more details about Whitehead groups.  In particular, Tits showed that $W(k,G)$ detects whether $G(k)$ is projectively simple.  Results of Margaux allow us to connect non-stable $K_1$--functors (as above) with Whitehead groups.  More precisely, one has the following result.

\begin{prop}
\label{prop:margaux}
Suppose $k$ is an infinite field, and $G$ is an isotropic reductive $k$--group (in the sense of \textup{Definition~\ref{defn:isotropic}}).  For any finitely generated separable extension $L/k$, there are canonical isomorphisms
\[
\pi_0(\Singaone G)(L)\cong W(L,G).
\]
functorial with respect to field extensions.  Moreover, the assignment $L \mapsto W(L,G)$ extends to a strongly $\aone$--invariant sheaf on $\Sm_k$.
\end{prop}

\begin{proof}
The first statement follows from Margaux \cite[Theorem 3.10]{Margaux} (see also Gille \cite[\S 4.3]{GilleBourbaki}) and only requires $G$ to be isotropic in the sense of Borel \cite[Definition V.20.1]{Borel}.  The second statement follows from the strong $\aone$--invariance of $a_{\Zar}\pi_0(\Singaone G)$ established in Theorem~\ref{thm:wha1invariance}(2).
\end{proof}

Whitehead groups are also related to arithmetic questions, e.g., regarding $R$--equivalence in $G(k)$ (see Gille \cite[\S 7]{GilleBourbaki} for a discussion of $R$--equivalence in the context under consideration).

\begin{cor}
\label{cor:structureofpi0}
Let $k$ be an infinite field and $G$ a semisimple simply-connected absolutely almost simple isotropic $k$--group, and set $\mathbf{G} := a_{\Zar}\pi_0(\Singaone G)$.  The following statements hold:
\begin{itemize}
\item[(i)] for any finitely generated separable extension $L/k$, there is an isomorphism of the form ${\mathbf G}(L) \cong G(L)/R$,
\item[(ii)] the contracted sheaf $\mathbf{G}_{-1}$ is trivial, and
\item[(iii)] if $k$ is furthermore perfect, and $G$ has classical type, then $\mathbf{G}$ is strictly $\aone$--invariant.
\end{itemize}
\end{cor}

\begin{proof}
The first statement follows from Proposition~\ref{prop:margaux} and \cite[Th\'eor\`eme 7.2]{GilleBourbaki}.

For the second statement, recall that $\mathbf{G}_{-1}(U) = \ker((id, 1)^*: \mathbf{G}(U\times\gm{}) \to \mathbf{G}(U))$.  As $\mathbf{G}$ is strongly $\aone$--invariant by Theorem~\ref{thm:wha1invariance}, $\mathbf{G}_{-1}$ is also strongly $\aone$--invariant by Morel \cite[Lemma 2.32]{MField}.  In particular, it is an unramified sheaf, which implies that the map $\mathbf{G}(X) \to \mathbf{G}(k(X))$ is injective for any irreducible smooth scheme $X$.  By \cite[Theorem 5.8]{GilleBourbaki}, we conclude that $\mathbf{G}(k(U)) \to \mathbf{G}(k(U \times \gm{}))$ is a bijection and thus that $\mathbf{G}_{-1}(U)$ is trivial, for any $U\in\Sm_k$.

For the final statement, if $k$ is furthermore perfect, it suffices by \cite[Theorem 5.46]{MField} to show that $\mathbf{G}$ is an abelian group valued functor.  Because $\mathbf{G}$ is unramified, it suffices to check abelianness on sections over extensions of the base field.  By Point (i), if $G$ has classical type, this follows from a result of Chernousov--Merkurjev \cite[Th\'eor\`eme 7.7]{GilleBourbaki}.
\end{proof}

\begin{rem}
The statement $\mathbf{G}_{-1} = 0$ of Corollary~\ref{cor:structureofpi0}(ii) is equivalent to the assertion that $\mathbf{G}$ is a birational sheaf.  If $G$ is not simply-connected, then the sheaf $\mathbf{G}$ is not, in general, birational.  For example suppose $G$ is a split semisimple group having non-trivial algebraic fundamental group $\Pi$ (in the sense of Chevalley groups).  If we let $\mathscr{H}^1_{\et}(\Pi)$ be the Nisnevich sheaf associated with the presheaf $U \mapsto H^1_{\et}(U,\Pi)$, then $\mathbf{G} \cong \mathscr{H}^1_{\et}(\Pi)$, which is not birational.
\end{rem}

\subsection{Nilpotence for non-stable $K_1$ functors}
\label{ss:nilpotence}
In this section, we include one more sample application of our results: we give a uniform proof of some nilpotence results for non-stable $K_1$--functors discussed in the previous section; such nilpotence results have been studied for instance by Bak \cite{BakNilpotent} and Bak--Hazrat--Vavilov \cite{BakHazratVavilov}.  The main result of this section is Theorem~\ref{thm:nilpotentbyabelian} which solves \cite[Problem 6]{BakHazratVavilov} in a number of cases of interest (see Remark~\ref{rem:bhvproblem6} for more details).  The approach we pursue has the benefit that it is conceptually simple (modelled on classical topological results) and applies to rather general isotropic reductive $k$--groups.  The tradeoff to this generality is that unlike \cite{BakHazratVavilov} we are forced to restrict attention to smooth $k$--algebras with $k$ an infinite field.

We use the following notation/terminology.   If $(\mathscr{X},x)$ is a pointed simplicial presheaf on $\Sm_k$, we will say that $\mathscr{X}$ is Nisnevich-connected if $a_{\Nis}\pi_0(\mathscr{X})$ is trivial and, given an integer $n \geq 1$, we will say that $\mathscr{X}$ is Nisnevich $n$--connected if $a_{\Nis} \pi_i(\mathscr{X},x)$ is trivial for $i \leq n$.

Now, suppose $\mathscr{G}$ is a simplicial presheaf of group-like $h$--spaces on $\Sm_k$ ($h$--group for short) pointed by the identity.  In that case, there is an induced morphism $\mathscr{G} \to a_{\Nis}\pi_0\mathscr{G}$; this morphism is a morphism of $h$--groups.  Write $\mathscr{G}^0$ for the homotopy fiber of $\mathscr{G}$, so that there is a homotopy fiber sequence of the form
\[
\mathscr{G}^0 \longrightarrow \mathscr{G} \longrightarrow a_{\Nis}\pi_0\mathscr{G}.
\]
By construction, $\mathscr{G}^0$ is a Nisnevich-connected $h$--group.  Using this notation, we can adapt arguments of Whitehead \cite[Corollary 2.12]{Whitehead} to establish an abstract nilpotence result.

\begin{prop}
\label{prop:nilpotence}
Assume $k$ is a Noetherian ring of finite Krull dimension, and suppose $\mathscr{G}$ is a Nisnevich-local simplicial presheaf of $h$--groups on $\Sm_k$ (pointed by the identity).
\begin{itemize}
\item[(i)] For any $X \in \Sm_k$, there is an exact sequence of groups of the form
\[
1 \longrightarrow [X,\mathscr{G}^0] \longrightarrow [X,\mathscr{G}] \longrightarrow a_{\Nis}\pi_0(\mathscr{G})(X).
\]
\item[(ii)] If $X \in \Sm_k$ has Krull dimension $\leq d$, then $[X,\mathscr{G}^0]$ is nilpotent of class $\leq d$.
\end{itemize}
\end{prop}

\begin{proof}
Point (i) is immediate from the long exact sequence of maps into a homotopy fiber sequence and the fact that $a_{\Nis}\pi_0(\mathscr{G})$ is $0$--truncated.

For Point (ii), it suffices to assume $\mathscr{G} = \mathscr{G}^0$ is Nisnevich-connected.  In that case, $\mathscr{G}^{\sma n}$ is Nisnevich $n$--connected. Indeed, this follows from the corresponding connectivity estimate for smash products of simplicial sets by checking on stalks.  Now, a straightforward obstruction theory argument (see Morel \cite[Lemma B.5]{MField}) using the connectivity estimate we just mentioned shows that $[X,\mathscr{G}^{\sma n}] = \ast$ if $\dim X \leq n$.  To conclude, we simply observe that every $n$--fold commutator in $[X,\mathscr{G}]$ factors as $X \to \mathscr{G}^{\sma n} \to \mathscr{G}$ (here, we use the assumption that $\mathscr{G}$ is an $h$--group and thus has a strict identity).
\end{proof}

\begin{rem}
The result above is rather general.  Indeed, as is evident from the proof, it holds for simplicial $h$--group objects in the local homotopy theory of simplicial presheaves on a site for which Postnikov towers converge.
\end{rem}

Now, suppose $G$ is an isotropic reductive $k$--group in the sense of Definition \ref{defn:isotropic}.  Following Petrov and Stavrova \cite{PetrovStavrova}, for any commutative $k$--algebra $R$ and any parabolic $k$--subgroup $P \subset G$, we define the elementary subgroup $E_P(R)$ as the subgroup of $G(R)$ generated by the $R$--points of the unipotent radical of $P$ and the $R$--points of the unipotent radical of its opposite.  A priori $E_P(R)$ depends on $P$ and $E_P(R)$ need not be a normal subgroup of $G$.  However, \cite[Theorem 1]{PetrovStavrova} guarantees that if each semi-simple normal subgroup of $G$ has rank $\geq 2$, then $E_P(R)$ is both independent of $P$ and normal in $G(R)$; under these hypotheses we define $E(R) := E_P(R)$ for any choice of proper parabolic and define $K^G_1(R) := G(R)/E(R)$.

We can also consider $G^0(R) \subset G(R)$, the subset of $G(R)$ consisting of matrices $g$ for which there exists $g(t) \in G(R[t])$ with $g(0) = 1$ and $g(1) = g$; this subgroup is evidently normal.  By construction $E_P(R) \subset G^0(R)$ and $KV^G_1(R) = G(R)/G^0(R)$.   Therefore there is a short exact sequence of groups
\[
1 \longrightarrow G^0(R)/E(R) \longrightarrow K^G_1(R) \longrightarrow KV^G_1(R) \longrightarrow 1.
\]

\begin{thm}
\label{thm:nilpotentbyabelian}
Suppose $k$ is an infinite field, $G$ is an isotropic reductive $k$--group in the sense of \textup{Definition~\ref{defn:isotropic}} and $R$ is a smooth $k$--algebra of dimension $d$.
\begin{enumerate}
\item[(i)] If for every finitely generated separable extension $L/k$ the Whitehead group $W(L,G)$ is trivial (abelian), then $KV^G_1(R)$ is (an extension of an abelian group by) a nilpotent group of class $\leq d$.
\item[(ii)] If furthermore $k$ is perfect, for every finitely generated separable extension $L/k$ the Whitehead group $W(L,G)$ is trivial (abelian), and every semi-simple normal subgroup of $G$ has rank $\geq 2$, then $K^G_1(R)$ is (an extension of an abelian group by) a nilpotent group of class $\leq d$.
\end{enumerate}
\end{thm}

\begin{proof}
	Let $\mathscr G=R_{\Zar}\Singaone G$.
According to Theorem~\ref{thm:wha1invariance}, the Nisnevich sheaf $a_{\Nis}\pi_0(\mathscr G)$ is strongly $\aone$--invariant.  By Proposition~\ref{prop:margaux} the group of sections $a_{\Nis}\pi_0(\mathscr G)(L)$ over finitely generated extensions $L/k$ coincides with $W(L,G)$.  In particular, the assumption that $W(L,G)$ is trivial (abelian) for every finitely generated separable extension $L/k$ implies that the sheaf $a_{\Nis}\pi_0(\mathscr G)$ is trivial (abelian).

By Theorem~\ref{thm:isotropicaonelocal} and Proposition~\ref{prop:A1naive}, $\mathscr G$ is Nisnevich-local and $KV^G_1(R)=[\Spec R,\mathscr G]$. Point (i) then follows immediately from Proposition~\ref{prop:nilpotence}.

Consider the exact sequence appearing before the statement gives a surjective map $K^G_1(R) \to KV^G_1(R)$.  Under the additional hypotheses in Point (ii), it follows immediately from a result of Stavrova \cite[Theorem 1.3]{Stavorvanonstable} that this surjection is an isomorphism and Point (ii) follows from Point (i).
\end{proof}

\begin{rem}
\label{rem:bhvproblem6}
Combined with known structural results about $W(-,G)$ (viewed as a functor on the category of finitely generated extensions of the base field), the above result solves a problem posed by Bak, Hazrat, and Vavilov \cite[Problem 6]{BakHazratVavilov} in a number of new cases.  For example, in \cite[Th{\'e}or{\`e}me 6.1]{GilleBourbaki}, Gille summarizes results of Chernousov--Platonov that detail situations where $W(-,G)$ is trivial for all finitely generated separable extensions $L/k$.  See Corollary \ref{cor:structureofpi0}(iii) for hypotheses that guarantee $W(-,G)$ is an abelian group valued functor on the category of (e.g., if $G$ has classical type).  Furthermore, it has been conjectured that $W(-,G)$ always takes values in abelian groups.
\end{rem}

\bibliographystyle{gtart}
\bibliography{bundles}

\end{document}